\documentclass[12pt]{amsart}
\usepackage{amsmath,amsthm,amsfonts,amssymb,eucal, hyperref}

\renewcommand{\>}{\rangle}

\renewcommand{\l}{\lambda}
\newcommand{\rR}{\mathrm{R}}

\newcommand{\T}{\Theta}

\newcommand{\oq}{\ {\raise 7pt\hbox{${\scriptstyle\circ}$}}
\kern -7pt{
\hbox{$Q$}}}

\newcommand{\R}{ \mathbb R}

\newcommand{\C}{ \mathbb C}
\newcommand{\N}{ \mathbb N}
\newcommand{\rbar}{{\overline{\mathbb R}}}

\renewcommand{\T}{\mathbb T}

\newcommand {\bk}{\mathbf k}

\newcommand {\bm}{\mathbf m}
\newcommand {\bl}{\mathbf l}
\newcommand {\bze}{\mathbf 0}

\newcommand {\br}{\mathbf r}

\newcommand {\bn}{\mathbf n}

\DeclareMathOperator{\Span}{{span}}

\newcommand{\1}
{{\,\vrule depth3pt height9pt}{\vrule depth3pt height9pt}
{\vrule depth3pt height9pt}{\vrule depth3pt height9pt}\,}

\DeclareMathOperator {\dist} {{dist}}

\DeclareMathOperator {\diam} {{diam}}

\DeclareMathOperator{\supp}{{supp}}

\newtheorem{thm}{Theorem}[section]
\newtheorem{cor}[thm]{Corollary}
\newtheorem{cla}[thm]{Claim}
\newtheorem{lem}[thm]{Lemma}
\newtheorem{prop}[thm]{Proposition}

\theoremstyle{definition}
\newtheorem{defn}[thm]{Definition}

\newtheorem{rem}[thm]{Remark}

\numberwithin{equation}{section}

\newcommand{\bee}{\begin{equation}}
\newcommand{\ene}{\end{equation}}
\newcommand{\bees}{\begin{equation*}}
\newcommand{\enes}{\end{equation*}}
\newcommand{\bes}{\begin{split}}
\newcommand{\ens}{\end{split}}

\newcommand{\bet}{\begin{thm}}
\newcommand{\ent}{\end{thm}}
\newcommand{\bel}{\begin{lem}}
\newcommand{\enl}{\end{lem}}
\newcommand{\bec}{\begin{cor}}
\newcommand{\enc}{\end{cor}}
\newcommand{\becl}{\begin{cla}}
\newcommand{\encl}{\end{cla}}
\newcommand{\bep}{\begin{proof}}
\newcommand{\enp}{\end{proof}}
\newcommand{\ber}{\begin{rem}}
\newcommand{\enr}{\end{rem}}
\newcommand{\ep}{\varepsilon}

\newcommand{\Z}{\mathbb Z}

\renewcommand{\l}{\left}
\renewcommand{\r}{\right}

\makeatletter
\def\square{\RIfM@\bgroup\else$\bgroup\aftergroup$\fi
  \vcenter{\hrule\hbox{\vrule\@height.6em\kern.6em\vrule}\hrule}\egroup}
\makeatother

 \usepackage{hyperref}
\def\level{\mathtt{level}}
\def\magn{\mathtt{magn}}

\newcommand{\one}{\mathbf{1}}

\usepackage[margin=1.1in]{geometry}

\begin{document}

\title[Perturbative localization for monotone potentials]
{Perturbative diagonalization and spectral gaps of quasiperiodic operators on $\ell^2(\Z^d)$ with monotone potentials}
\author[I. Kachkovskiy]
{Ilya Kachkovskiy}
\address{Department of Mathematics\\ Michigan State University\\
Wells Hall, 619 Red Cedar Rd\\ East Lansing, MI\\ 48824\\ USA}
\email{ikachkov@msu.edu}
\author[L. Parnovski]
{Leonid Parnovski}
\address{Department of Mathematics\\ University College London\\
Gower Street\\ London\\ WC1E 6BT\\ UK}
\email{l.parnovski@ucl.ac.uk}
\author[R. Shterenberg]
{Roman Shterenberg}
\address{Department of Mathematics\\ University of Alabama, Birminghan\\
Campbell Hall\\1300 University Blvd\\ Birmingham, AL\\ 35294\\USA }
\email{shterenb@math.uab.edu}



\date{\today}



\begin{abstract}
We obtain a perturbative proof of localization for quasiperiodic operators on $\ell^2(\Z^d)$ with one-dimensional phase space and monotone sampling functions, in the regime of small hopping. The proof is based on an iterative scheme which can be considered as a local (in the energy and the phase) and convergent version of KAM-type diagonalization, whose result is a covariant family of uniformly localized eigenvalues and eigenvectors. We also proof that the spectra of such operators contain infinitely many gaps.
\end{abstract}
\maketitle
\section{Introduction}
\subsection{Statements of the results}Let $\alpha\ge 1$ and $f\colon\R\to  [-\infty,+\infty)$. We will say that $f$ is {\it $\alpha$-H\"older monotone} if $f$ is $1$-periodic, and
$$
f(y)-f(x)\ge (y-x)^{\alpha},\quad \text{for}\quad 0\le x\le y<1.
$$
For $\alpha=1$, we will call such functions {\it Lipschitz monotone}. In this paper, we consider the following quasiperiodic operator family on $\ell^2(\Z^d)$:
\bee
\label{eq_h_def}
(H(x)\psi)(\bn)=\ep (\Delta \psi)(\bn)+f(x+\omega\cdot\bn)\psi(\bn),\quad \bn=(n_1,\ldots,n_d)\in\Z^d.
\ene
Here $\Delta$ is the discrete nearest neighbor Laplace operator on $\Z^d$:
$$
(\Delta \psi)(\bn)=\sum\limits_{\bm\in \Z^d\colon|\bm-\bn|=1}\psi(\bm),\quad |\bn|=|\bn|_1=|n_1|+|n_2|+\ldots+|n_d|.
$$
We will consider the following class of weakly Diophantine frequency vectors for $\rho>0$, $\mu>0$:
\bee
\label{eq_liouville_def}
\Omega_{\rho,\mu}:=\l\{\omega\in [0,1)^d\colon \|\bn\cdot\omega\|:=\dist(\bn\cdot\omega,\Z)\ge e^{-\rho|\bn|^{\frac{1}{1+\mu}}},\bn\neq \bze\r\}.
\ene
Let $\{e_{\bn}\colon\bn\in \Z^d\}$ be the standard basis of $\ell^2(\Z^d)$. For vector-valued functions and operators, we will use the following notation in order to denote their components:
$$
\psi(x;\bn):=\langle\psi(x),e_{\bn}\rangle,\quad H(x;\bm,\bn):=\langle H(x)e_{\bn},e_{\bm}\rangle.
$$
\begin{thm}
\label{th_main}Let $\alpha\ge 1$, $0<\delta<1$, $\rho,\mu>0$. There exists $\ep_0=\ep_0(d,\rho,\mu,\alpha,\delta)>0$ such that for every $\ep\in (0,\ep_0)$, $\omega\in \Omega_{\rho,\mu}$, and $\alpha$-H\"older monotone function $f\colon \R\to [-\infty,+\infty)$ one can find a $1$-periodic function $E\colon \R\to[-\infty,+\infty)$, strictly increasing on $[0,1)$, and a $1$-periodic measurable function $\psi\colon [0,1)\to \ell^2(\Z^d)$ such that
\bee
\label{eq_sinai_eigenfunction}
H(x)\psi(x)=E(x)\psi(x),\quad\forall x\in \R;
\ene
\bee
\|\psi(x)\|_{\ell^2(\Z^d)}=1,\quad |\psi(x,\bze)-1|<\ep^{1-\delta};\quad |\psi(x;\bn)|\le \ep^{(1-\delta)|\bn|}\quad \text{for}\quad |\bn|\neq 0.
\ene
If $f$ is Lipschitz monotone, then $E$ is also Lipschitz monotone.
\end{thm}
\noindent Define
$$
(T^{\bn}\psi)(x;\bm):=\psi(x;\bn+\bm).
$$
In a standard way, the conclusion of Theorem \ref{th_main} implies Anderson localization for the operator family $H(x)$, as follows:
\begin{cor}
\label{cor_localization}
Under the assumptions of Theorem $\ref{th_main}$, for every $x\in \R$ one has
\bee
\label{eq_covariance}
H(x)(T^{\bn}\psi(x-\bn\cdot\omega))=E(x-\bn\cdot\omega)(T^{\bn}\psi(x-\bn\cdot\omega)).
\ene
The corresponding family
\bee
\label{eq_eigenvectors}
\psi_{\bn}(x):=T^{\bn}\psi(x-\bn\cdot\omega),\quad \bn\in\Z^d,
\ene
is an orthonormal basis of eigenvectors of $H(x)$, with eigenvalues 
\bee
\label{eq_eigenvalues}
E_{\bn}(x):=E(x-\bn\cdot\omega),
\ene
respectively. In particular, the spectrum of $H(x)$ is pure point and simple, with uniformly exponentially localized eigenfunctions.
\end{cor}
\begin{rem}
\label{rem_infinite}
We allow $f(x)=-\infty$ for $x=0$. Similarly to \cite{JK2}, this corresponds to the case of infinite coupling, where the operator has an eigenvector $e_{\bze}$ with a (generalized) eigenvalue $-\infty$. We refer the reader to \cite[Section 3]{Cantor} for more details on Schr\"odinger operators with infinite potentials. It is easy to check that the proof of Theorem \ref{th_main} extends to the case $f(0)=-\infty$, with infinite denominators in certain places leading to the corresponding expressions vanishing. Instead of $\alpha$-H\"older monotone functions on $[0,1)$, one can also consider $\alpha$-H\"older monotone functions on $(0,1]$ with values in $(\infty,+\infty]$ (for example, $f(x)$ can be replaced by $-f(-x)$).
\end{rem}
A natural question for quasiperiodic Schr\"odinger operators is the existence of spectral gaps. We will say that $f\colon \R\to \R$ is {\it sawtooth-type}\footnote{Note that the definition here is somewhat broader than the corresponding one in \cite{Cantor}, where it is assumed that $f$ is Lipschitz monotone.} if $f$ is $\alpha$-H\"older monotone, continuous on $[0,1)$, and extends to a continuous map from $[0,1]$ to $\R$ (in particular, this implies that $f$ is bounded). For $d=1$, it was shown in \cite{Cantor} that a large class of quasiperiodic operators with sawtooth-type potentials with $\alpha=1$ has Cantor spectrum. In order to state the second main result, suppose that we are in the setting of Theorem \ref{th_main}. Define
$$
\ell^2_{\bze}(\Z^d):=\overline{\Span\{\psi_{\bn}(0)\colon \psi_{\bn}(0;\bze)\neq 0\}}_{\ell^2(\Z^d)}=\overline{\Span\{(H(0)-z)^{-1}e_{\bze}\colon z\in \C\setminus\R\}}_{\ell^2(\Z^d)}
$$
be the cyclic subspace of $H(0)$ associated to the vector $e_{\bze}$ (that is, the subspace spanned by all eigenvectors \eqref{eq_eigenvectors} of $H(0)$ that do not vanish at the origin). Note that, similarly to \cite{Cantor}, the operator at the phase $x=0$ plays a special role, and it is important that the conclusion of Theorem \ref{th_main} includes $x=0$.  It is easy to see that, for $\ep\neq 0$, the space $\ell^2_{\bze}(\Z^d)$ is infinite-dimensional. Let
$$
H_{\bze}:=H(0)|_{\ell^2_{\bze}(\Z^d)}
$$
be the corresponding restriction of the operator $H(0)$. The following is our second main result.
\begin{thm}
\label{th_gaps}
Suppose $f$ is sawtooth-type, $\omega\in \Omega_{\rho,\mu}$, and $\ep$ is small enough so that the conclusion of Theorem $\ref{th_main}$ holds. Then $\sigma(H_{\bze})$ is a Cantor set, and $\sigma(H(x))$ contains infinitely many gaps.
\end{thm}

\subsection{Overview}
The most well-known example of a (Lipschitz) monotone potential is the Maryland model with $f(x)=\cot(\pi x)$, see \cite{Pastur,Grempel,JHY,JL,JY,Simon_maryland}, for which localization is known for all $\ep>0$ and Diophantine $\omega$ (most likely, the proofs can be extended to $\omega\in\Omega_{\rho,\mu}$). For $d=1$, a large class of operators with Lipschitz monotone potentials also admits Anderson localization for all $\ep>0$, see \cite{JK1,Ilya,Xiaowen,JK2}. In the case of general monotone potentials for $d>1$, other than the Maryland model, only perturbative methods are currently available. The situation depends significantly depending on whether or not $f$ has discontinuities in the generalized sense (that is, considered as a map from $[0,1)$ identified with a circle to $\R\cup\{\infty\}$ also identified with a circle). For Maryland-type potentials (that is, without additional discontinuities), under various assumptions on regularity, perturbative localization can be obtained by a variety of methods \cite{Bellissard,KPS,KKPS,Facchi,Shi1}. However, the presence of discontinuities such as in $f(x)=\{x\}$ brings some difficulties to the previously known perturbative methods, due to the lack of regularity of the small denominators and/or incomplete cancellations due to discontinuities. Some partial results \cite{Craig_pure_point,Facchi,Pos,Shi1} were obtained by carefully choosing the off-diagonal part of the operator. In the recent work \cite{Shi}, perturbative localization for a class of Lipshitz monotone potentials on $\Z^d$ was obtained, with an approach involving multi-scale analysis and studying eigenvalues of finite volume operators using Rellich functions.

Theorem \ref{th_main} of the present paper is a direct perturbative result on Anderson localization for H\"older monotone potentials. Although, on a very general level, one can find some parallels between the difficulties both our proof and \cite{Shi} address, the two proofs are quite different and have been obtained independently and almost sumultaneously. The following are some aspects of proofs of Theorems \ref{th_main} and \ref{th_gaps}.
\begin{itemize}
	\item Our main result is, essentially, convergence of a diagonalization scheme for the operator $H(x)$ which, like the perturbation series in \cite{KPS}, is quite naive: we are trying to diagonalize the operator by a sequence of perturbative Jacobi rotations. The usual Nash--Moser iteration scheme of this kind is known to have difficulties in the case of potentials with discontinuities (see, for example, \cite{Shi1}). The main idea is to temporarily give up on the goal of a complete diagonalization and, instead, provide a partial diagonalization, which, ultimately, will isolate a one-dimensional eigenspace. This allows more freedom in choosing the matrices that deliver partial diagonalization and significantly reduces the number of small denominators that one has to control simultaneously.
	\item Since the original function $f$ is (quantitatively) monotone, it is easy to control small denominators at the first step of the iterative procedure. However, the procedure does not preserve exact monotonicity. While some kind of approximate monotonicity can be preserved, its accuracy on the interval under consideration has to improve, in order to be able to treat smaller and smaller denominators. It seems difficult to obtain such an improvement by means of an inductive procedure alone. This difficulty is resolved by relating the diagonal entries to the branches of the eigenvalues of the finite volume operators. In our approach, the construction of monotone eigenvalue branches is based on a combinatorial argument, originally obtained in \cite{Ilya}, which determines a monotone eigenvalue branch based on the eigenvalue counting function. The partial monotonicity that is preserved during the course of the inductive procedure provides sufficient information to confirm that the diagonal entry under consideration is close to the ``correct'' monotone branch of the finite volume eigenvalue.
	\item Our construction is ``truly perturbative'' in the sense that we obtain the basis of eigenvectors that is a small perturbation of the standard basis. It also produces covariant branches of eigenvalues and eigenvectors defined for all phases (``Sinai functions'' \cite{Sinai}) and naturally implies simplicity of the eigenvalues.
	\item The condition $\omega\in \Omega_{\rho,\mu}$ is weaker than the usual Diophantine condition and is convenient to work with in the chosen scaling for the small denominators. The H\"older monotonicity condition also does not present substantial extra difficulties in the main proof of localization. However, the main issue is that, if a potential satisfies a quantitative monotonicity condition, then the monotone eigenvalue branches only inherit a local condition of the same kind. While local Lipschitz monotonicity implies global Lipschitz monotonicity, the same does not hold for H\"older monotonicity. As a result, we lose quantitative control on monotonicity of the eigenvalue branches and, ultimately, the function $E(\cdot)$. From our analysis alone, we cannot exclude the possibility of $E$ being constant on an interval. However, since in all cases we have (non-strict) monotonicity of $E$ on $[0,1)$, the presence of a flat piece would lead to a jump discontinuity in the integrated density of states, which is impossible \cite{CS_ids,BK}. As a consequence, $E(\cdot)$ is strictly monotone on $[0,1)$ in all cases covered by Theorem \ref{th_main}.
	\item In the case of $\alpha$-H\"older monotone functions with $\alpha>1$, the non-perturbative arguments of \cite{JK1,Ilya,Xiaowen,JK2} encounter some difficulties: while almost everywhere positivity of the Lyapunov exponent still holds, it is harder to obtain a quantitative positivity argument and/or absolute continuity of the IDS, which is originally used to state that the possible zero set of the Lyapunov exponent does not contribute to the spectral measure.
	\item In Theorem \ref{th_gaps}, note that, for sawtooth-type $f$, the spectrum $\sigma(H(x))$ does not depend on $x$. The proof of Theorem \ref{th_gaps} is contained in Section \ref{section_gaps}. The first claim immediately follows from del Rio--Gordon--Makarov--Simon theorem applied to operators of the form $H_{\bze}+t\<e_{\bze},\cdot\>e_{\bze}$, similarly to the proof of \cite[Theorem 2.9]{Cantor}. For topologically generic $t\in \R\setminus[f(0),f(1-0)]$, infinitely many gaps in $\sigma_{\mathrm{ess}}(H_{\bze})+t\<e_{\bze},\cdot\>e_{\bze}$ will contain an (isolated) eigenvalue. The gap filling theorem \cite[Theorem 1.2]{Cantor} and simplicity of the spectrum imply that all these gaps will be present in $\sigma(H(x))$.
\end{itemize}
\subsection{Structure of the paper}In Section 2, we describe the iterative scheme and state Theorem \ref{th_convergence}, the main technical result, which is the convergence of the scheme to correct eigenvalues and eigenvectors. The ``combinatorial'' structure of the diagonalizing matrices is described in Lemma \ref{lemma_W_support}.

In Section 3, we state and prove (most of) Theorem \ref{th_inductive_estimates}, which contains the main quantitative bounds that imply the convergence of the scheme. As mentioned in the overview, the proof requires an additional argument that involves finite volume eigenvalue counting functions, Theorem \ref{th_ind_4}, which is proved in Section 4.

In Section 5, we finish the proof of Theorem \ref{th_main} and establish monotonicity of the function $E(\cdot)$. We also discuss its relation with the integrated density of states.

In Section 6, we prove Theorem \ref{th_gaps}, using monotonicity of $E(\cdot)$ obtained in Theorem \ref{th_main} and ``soft'' arguments from \cite{Cantor}, including the gap filling theorem.
\subsection{Acknowledgements}I. K. was supported by the NSF grants DMS--1846114, DMS--2052519, and the 2022 Sloan Research Fellowship. L. P. was supported by the EPSRC grant EP/V051636/1. R. S. was supported by the NSF grant DMS-2306327.

The authors would like to thank S. Jitomirskaya for valuable discussions and the anonymous referees for useful suggestions and remarks.
\subsection{Data availability statement.}No data was produced during the research and preparation of the manuscript.
\subsection{Conflict of interest statement.}The authors declare that they have no conflicts of interests and no financial interests related to the present manuscript and the associated research.

\section{An iterative construction}
\subsection{Overview}
Fix $x_0\in [0,1)$. Let $H(x)$ be defined as in \eqref{eq_h_def}. Our goal will be to construct a sequence of self-adjoint operators $H^{(s)}(x,x_0)$, with $H^{(0)}(x,x_0)=H(x)$, and a sequence of unitary operators $U^{(s)}(x,x_0)$ with 
$$
H^{(s+1)}(x,x_0)=U^{(s+1)}(x,x_0)^{-1}H^{(s)}(x,x_0)U^{(s+1)}(x,x_0),\quad s=0,1,2,\ldots.
$$
Each operator $U^{(s)}$ will be a small perturbation of $\one$:
\bee
\label{eq_norm_estimate_u}
\|U^{(s)}-\one\|\le \ep^{s(1-\delta/10)},
\ene
As a consequence, the sequence
$$
W^{(s)}(x,x_0)=\prod_{\ell=1}^s U^{(\ell)}(x,x_0)
$$
will converge (in the operator norm) to a unitary operator $W^{(\infty)}(x,x_0)$, and we will make sure that the operator $H^{(\infty)}(x,x_0)=W^{(\infty)}(x,x_0)^{-1}H(x)W^{(\infty)}(x,x_0)$ will satisfy
$$
H^{(\infty)}(x_0,x_0;\bze,\bm)=0, \quad \forall \bm\in \Z^d\setminus\{\bze\}.
$$
In other words, conjugation by $W^{(s)}(x_0,x_0)$ will gradually eliminate the off-diagonal entries of $H(x_0;\bze,\bn)$ with $\bn\neq \bze$.

It is easy to see that the above construction produces a normalized eigenvector
\bee
\label{eq_psi_definition}
\psi(x_0):={W^{(\infty)}(x_0,x_0)e_{\bze}}
\ene
with the eigenvalue
\bee
\label{eq_e_definition}
E(x_0):=H^{(\infty)}(x_0,x_0;\bze,\bze),
\ene
that is,
$$
H(x_0)\psi(x_0)=E(x_0)\psi(x_0).
$$
From \eqref{eq_norm_estimate_u} and the structure of $U^{(s)}$, this eigenvector will be uniformly (in $x_0$) localized near the origin:
\bee
\label{eq_psi_localized}
\|\psi(x_0)-e_{\bze}\|\le \ep^{1-\delta},\quad |\psi(x_{\bze};\bn)|\le \ep^{(1-\delta)|\bn|}.
\ene
As stated in Corollary \ref{cor_localization}, the covariance property
$$
H(x_0+\bn\cdot\omega)\psi(x_0+\bn\cdot\omega)=T^{\bn}H(x_0)T^{-\bn}\psi(x_0+\bn\cdot\omega),\quad \text{where}\quad (T^{\bn}\psi)(\bm)=\psi(\bm+\bn),
$$
implies
$$
H(x_0)T^{\bn}\psi(x_0-\bn\cdot\omega)=E(x_0-\bn\cdot\omega)T^{\bn}\psi(x_0-\bn\cdot\omega).
$$
In other words, $\psi_{\bn}(x_0)=T^{\bn}\psi(x_0-\bn\cdot\omega)$ is another eigenvector of $H(x_0)$, exponentially localized near $\bn$. The standard arguments imply that the isometric operator
$$
\Psi(x_0)\colon\ell^2(\Z^d)\to\ell^2(\Z^d),\quad \Psi(x_0) e_{\bn}=\psi_{\bn}(x_0)=T^{\bn}\psi(x_0-\bn\cdot\omega),
$$
is a small perturbation of $\one$, and therefore is a unitary operator. As a consequence, the above construction provides a complete orthonormal basis of exponentially localized eigenvectors of $H(x_0)$.

The sequence of operators $U^{(s)}(x,x_0)$ will be determined by $H(x)$ and a sequence of intervals
\bee
\label{eq_interval_general}
\T=I_0\supset I_1\supset  I_2\supset\ldots
\ene
Here, an interval is an open connected subset of $\T$. We will always assume that $\T$ is identified with $[0,1)$. An interval may appear as a disconnected subset of $[0,1)$ in the standard topology: for example, for $x_0=4/5$, we have
$$
(x_0-2/5,x_0+2/5)=[0,1/5)\cup(2/5,1)\,\,\mathrm{mod}\,\,\Z.
$$
Each operator $U^{(s)}(x,x_0)$ will depend on $x_0$ through the choice of the intervals $I_1,I_2,\ldots,I_{s}$:
\bee
\label{eq_u_actual_dependence}
U^{(s)}(x,x_0)=U^{(s)}(I_1(x_0),I_2(x_0),\ldots,I_s(x_0),x),
\ene
where
\bee
\label{eq_interval_specific}
I_s=I_s(x_0)=(x_0-\beta^{(s)},x_0+\beta^{(s)}),\quad \beta^{(s)}\to 0\,\,\,\text{as}\,\,\,s\to \infty,
\ene
and
\bee
\label{eq_betas_def}
\beta^{(s)}:=\beta^{\frac{s}{\log(s+1)}},\quad 0<\beta<\beta_0.
\ene
The choice of a small $\beta_0$ in relation to other small and large parameters will be discussed later in Remark \ref{rem_small_parameters}. We will also have
$$
U^{(s)}(x,x_0)=\one ,\quad \text{for}\quad x\in \T\setminus I_s.
$$
\subsection{Preliminaries: perturbative Jacobi rotation and supports of operators} 
\begin{defn}
A self-adjoint matrix
\bee
\label{eq_matrix}
\begin{pmatrix}
	a&h\\ h &b
\end{pmatrix},\quad a,b,h\in \R
\ene
is {\it diagonally dominant} if
\bee
\label{eq_diagonally_dominant}
|h|<|b-a|.
\ene	
\end{defn}
For $h\in (-|b-a|,|b-a|)$, there exist unique real analytic (in $h$) branches of eigenvectors of \eqref{eq_matrix} which become the standard basis vectors at $h=0$. It is easy to see (by rescaling) that these eigenvectors only depend on $\frac{h}{|b-a|}$.  More precisely, there exists a unique function $\tau\mapsto U(\tau)$ such that $U(\tau)$ is an orthogonal $(2\times 2)$-matrix, $U(0)=\one$, the entries of $U(\tau)$ are real analytic in $\tau$ on $(-1,1)$, and the matrix
$$
U^{-1}(\tau)\begin{pmatrix}
	a&h\\ h &b
\end{pmatrix}U(\tau),\quad\text{where}\quad \tau=\frac{h}{|b-a|}\in (-1,1),
$$
is diagonal. The actual expression for $U(\tau)$ can be easily written down but is not particularly convenient to work with due to the presence of square roots. In our situation, the case of specific importance will be that of $\tau\ll 1$, in which case it is often be sufficient to work with the lowest non-trivial approximation in $\tau$:
\bee
\label{eq_u_structure}
U(\tau)=\begin{pmatrix}
1&\tau\\-\tau&1
\end{pmatrix}+\tau^2 W(\tau)=:V(\tau)+\tau^2 W(\tau),
\ene
where $W(\cdot)$ is a real analytic matrix-valued function with, say, $\|W(\tau)\|<10$ for $|\tau|<1/10$. The general case (for matrices that are not necessarily diagonally dominant) of the above procedure is usually referred to as {\it Jacobi rotation}. We will be applying this construction in the following somewhat more specific situation.
\begin{defn}Suppose that 
$$
A(x)=\begin{pmatrix}
	a(x)&h(x)\\ h(x) &b(x)
\end{pmatrix},\quad\text{where}\quad  a,b\colon \R\to [-\infty,+\infty),\quad h\colon \R\to \R
$$
is a $1$-periodic self-adjoint matrix-valued function. Assume that $[0,1)$ is identified with $\T$ and $I\subset \T$ is an interval such that $A(x)$ is diagonally dominant for $x\in I$. Define 
\bee
\label{eq_delta_definition}
\tau(x):=\frac{h(x)}{|b(x)-a(x)|}
\ene
as above. The matrix-valued function
$$
U_I(x):=\begin{cases}
U(\tau(x)),&x\in I+\Z;\\
\one,& x\in \R\setminus (I+\Z)
\end{cases}
$$
will be called {\it the perturbative Jacobi rotation matrix for $A(x)$ on the interval $I$}. The map $A(\cdot)\mapsto U_I(\cdot)^{-1}A(\cdot)U_I(\cdot)$ will be called {\it the perturbative Jacobi rotation on $I$}.
\end{defn}
\begin{rem}
\label{rem_standard_infinity}
In the above construction, we allow the possibility of (at most one of) the entries $a(x)$, $b(x)$ being equal to $-\infty$. For these values of $x$, we have $U_I(x)=\one$. Note that, in this construction, we do not use any smoothing cut-off functions to interpolate between $U(\tau(x))$ and $\one$. 
\end{rem}

The construction of the diagonalizing matrices will involve a sequence of perturbative Jacobi rotations applied to various $(2\times 2)$-blocks of the operator $H^{(j)}(x,x_0)$. In order to describe the relations between these diagonalizations, we will need some combinatorial preliminaries.

\begin{defn}Let $A$ be a bounded normal operator on $\ell^2(\Z^d)$ and $S\subset \Z^d$. We will say that $A$  {\it is supported on} $S$ if $A=A_S\oplus \one_{\Z^d\setminus S}$, where $A_S$ is some (normal) operator on $\ell^2(S)$.	
\end{defn}
In other words, $\ell^2(S)$ is an invariant subspace of $A$, and $A$ acts as an identity on $\ell^2(S)^{\perp}$. For every (bounded) operator $A$ on $\ell^2(\Z^d)$, there exists a minimal (with respect to inclusion) subset $S\subset\Z^d$ with this property, which we will denote by $\supp\,A$. An operator supported on $S$ can be naturally identified with on operator on $\ell^2(S)$. The notion of support will be used for unitary operators, which is why it is natural to extend them by identity outside of that.

\begin{defn}Let $A$ be a bounded normal operator on $\ell^2(\Z^d)$ and $\{S_1,S_2,\ldots\}$ be a collection of mutually disjoint subsets of $\Z^d$. Assume that
$$
A=\l(\oplus_r A_r\r)\oplus \one_{\Z^d\setminus(\cup_r S_r)},
$$
where $A_r$ is some (normal) operator on $\ell^2(S_r)$. In this case we will say that $A$ {\it is supported on $\{S_1,S_2,\ldots\}$.}
\end{defn}
\begin{defn}Let $A$ be a bounded normal operator on $\ell^2(\Z^d)$. We say that $S\subset \Z^d$ is a {\it component of the support of $A$} if the following are true:
\begin{enumerate}
	\item $\ell^2(S)$ is an invariant subspace for $A$.
	\item $A|_{\ell^2(S)}\neq \one_{\ell^2(S)}$.
	\item $S$ is a minimal (with respect to inclusion) non-empty subset of $\Z^d$ satisfying properties (1) and (2).
\end{enumerate}
\end{defn}
Note that any two distinct components of the support of $A$ must be disjoint. If $\{S_1,S_2,\ldots\}$ is the collection of all components of the support of $A$, then $A$ is supported on $\{S_1,S_2,\ldots\}$, making it a somewhat ``most refined'' choice of subsets $S_j$. Since $A$ is normal, a subspace $\ell^2(S)$ is invariant for $A$ if and only if its orthogonal complement $\ell^2(S)^{\perp}=\ell^2(\Z^d\setminus S)$ is invariant for $A$.



Suppose that $A^{(i)}$ is supported on $\{S_1^{(i)},S_2^{(i)},\ldots\}$, for $i=1,2$, respectively. Then one can describe the support of $A=A^{(2)}A^{(1)}$ as follows. Consider the graph $\Gamma$ with the set of vertices $\Z^d$, and the property that there is an edge between $\bm$ and $\bn$ iff, for some $p,q\in \N$, we have $\bm\in S^{(1)}_p$, $\bn\in S^{(2)}_q$, $S^{(1)}_p\cap S^{(2)}_q\neq \varnothing$, or vice versa (that is, with $\bm$ and $\bn$ interchanged). Then $A$ is supported on $\{S_1,S_2,\ldots\}$, where $S_k$ are connected components of that graph. More directly, the sets $S_k$ are equivalence classes of the following equivalence relation: $\bm\sim\bn$ if there exists a finite sequence $S_{k_1}^{(i_1)},\ldots,S_{k_p}^{(i_p)}$ with $\bm\in S_{k_1}^{(i_1)}$, $\bn\in S_{k_p}^{(i_p)}$, $S_{k_r}^{(i_r)}\cap S_{k_{r+1}}^{i_{(r+1)}}\neq \varnothing$. Without loss of generality, one can assume $i_{r+1}\neq i_r$ (note that both indices only take values $1$ or $2$). One can also extend this description inductively to products of any finite number of operators.

A special case of the above construction will appear in our situation.
\begin{lem}
\label{lemma_separated}
In the above setting, suppose that $A^{(i)}$ satisfy the following property: for every $p,q,r\in \N$ such that
$$
S^{(2)}_p\cap S^{(1)}_r\neq \varnothing,\quad S^{(2)}_q\cap S^{(1)}_r\neq \varnothing,
$$
one must have $p=q$.

Then $A=A^{(2)}A^{(1)}$ is supported on $\{S_1,S_2,\ldots,S_1',S_2',\ldots\}$, where
$$
S_j=S_j^{(2)}\cup \l(\bigcup_{r\colon S_r^{(1)}\cap S_j^{(2)}\neq \varnothing}S_r^{(1)}\r),
$$
and $S_j'$ are those of $S_r^{(1)}$ that do not intersect any of $S_p^{(2)}$.
\end{lem}
\begin{proof}
This follows from the fact that any chain $S_{k_1}^{(i_1)},\ldots,S_{k_p}^{(i_p)}$ with $\bm\in S_{k_1}^{(i_1)}$, $\bn\in S_{k_p}^{(i_p)}$, $S_{k_r}^{(i_r)}\cap S_{k_{r+1}}^{i_{(r+1)}}\neq \varnothing$ can include at most one subset from $\{S_1^{(2)},S_2^{(2)},\ldots\}$, and therefore can be reduced to a chain of length at most $3$.
\end{proof}
Lemma \ref{lemma_separated} essentially states that, if the distances between components of $\supp A^{(2)}$ are larger than the diameters of the components of $\supp A^{(1)}$, then the components of $\supp A$ will be either located near those of $A^{(2)}$, absorbing some of $\supp A^{(1)}$, or be the remaining ``isolated'' components of $\supp A^{(1)}$.
\label{eq_product}
\subsection{The construction and the main result}\label{subsection_construction}
For an operator $A$ on $\ell^2(\Z^d)$ and a finite subset $\Lambda\subset \Z^d$, we will use the notation
$$
A_{\Lambda}:=\one_{\Lambda}A \one_{\Lambda}|_{\ell^2(\Lambda)}
$$
to define the ``finite volume restriction'' of $A$. 
As a reminder, we start with a sequence of intervals
$$
\T=I_0\supset I_1\supset I_2\ldots,\quad I_s=(x_0-\beta^{(s)},x_0+\beta^{(s)})
$$
and an operator $H^{(0)}(x,x_0)=H(x)$. We will be constructing a sequence of unitary operators
$$
U^{(s)}(x,x_0),\quad s=1,2,3,\ldots,
$$
and 
$$
W^{(s)}(x,x_0)=\prod_{\ell=s}^{1}U^{(s)}(x,x_0)=U^{(s)}(x,x_0)U^{(s-1)}(x,x_0)\ldots U^{(1)}(x,x_0).
$$
The operators $U^{(s+1)}$ will be determined from $H^{(s)}$ defined by
\begin{multline*}
H^{(s+1)}(x,x_0)=U^{(s+1)}(x,x_0)^{-1}H^{(s)}(x,x_0)U^{(s+1)}(x,x_0)=\\=W^{(s+1)}(x,x_0)^{-1}H(x)W^{(s+1)}(x,x_0).	
\end{multline*}

The following is the construction of the diagonalizing operators $W^{(s)}(x,x_0)$, described step-by-step.

\begin{itemize}
	\item[(diag1)] Let $\bm\in \Z^d\setminus\{\bze\}$ and $H^{(s)}_{\{\bze,\bm\}}(x,x_0)$ be the $(2\times 2)$-block of the operator $H^{(s)}(x,x_0)$. Let $U^{(s+1)}_{\{\bze,\bm\}}(x,x_0)$ be the perturbative Jacobi rotation matrix of $H^{(s)}_{\{\bze,\bm\}}(x,x_0)$ on $I_{s+1}$ (note that this requires $H^{(s)}_{\{\bze,\bm\}}(x,x_0)$ to be diagonally dominant for $x\in I_s$, which will be verified later). Let also
	$$
	U^{(s+1),\bm}(x,x_0):=U^{(s+1)}_{\{\bze,\bm\}}(x,x_0)\oplus \one_{\Z^d\setminus\{\bze,\bm\}}
	$$
	be the extension of $U^{(s+1)}_{\{\bze,\bm\}}(x,x_0)$ by identity into $\ell^2(\Z^d)$. Clearly, $U^{(s+1),\bm}(x,x_0)$ is supported on $\{\bze,\bm\}$.
	\item[(diag2)] Let 
$$
	U_0^{(s+1)}(x,x_0):=\prod_{\bm\colon 0<|\bm|\le s+1}U^{(s+1),\bm}(x,x_0).
$$
We will assume that some choice of order of multiplication, independent of $x$ and $x_0$, is fixed for each $s$. For example, one can use lexicographic order. The main result does not depend on this choice.

\item[(diag3)] The matrix function $U_0^{(s+1)}(x,x_0)$ constructed above is not quasiperiodic. Consider the following quasiperiodic extension using the translation operator:
$$
U^{(s+1)}(x,x_0):=\prod_{\bn\in \Z^d}T^{-\bn}U_0^{(s+1)}(x+\omega\cdot\bn,x_0)T^{\bn},\quad (T^{\bn}f)(\bm)=f(\bn+\bm).
$$
Let
\bee
\label{eq_R_def}
{\mathrm{R}}_{\bn}^{(s+1)}(x,x_0):=\begin{cases}
\{\bm\in \Z^d\colon |\bm-\bn|\le s+1\},&x+\bn\cdot\omega\in I_{s+1};\\
\varnothing,&x+\bn\cdot\omega\in \T\setminus I_{s+1}.
\end{cases}
\ene
For sufficiently small $\beta$ (see Remark \ref{rem_small_parameters}), the sets $\rR_{\bn}^{(s)}(x,x_0)$ for different $\bn$ are disjoint. In this case we have that $U^{(s+1)}(x,x_0)$ is supported on
$$
\{{\mathrm{R}}_{\bn}^{(s+1)}(x,x_0)\colon x+\bn\cdot\omega\in I_{s+1}\},
$$
\item[(diag4)] Finally, as stated earlier, define
\bee
\label{eq_w_definition}
W^{(s+1)}(x,x_0):=\prod_{\ell=s+1}^{1} U^{(\ell)}(x,x_0)=U^{(s+1)}(x)W^{(s)}(x,x_0).
\ene
\end{itemize}
\begin{rem}
\label{rem_small_parameters}
Note that the operators \eqref{eq_w_definition} are completely determined by $d,f,\omega$, and $\beta$.

Their existence will rely several inductive estimates which will involve several parameters: $\alpha$, $\delta$, $\rho$, $\mu$, $\beta_0$, $\beta$, $\ep$, $\ep_0$, and $M$ (introduced in Subsection \ref{subsection_magnitude}). It is convenient to fix the specific order in which they are chosen in the proof.
\begin{enumerate}
	\item Initially, $\rho>0$, $\mu>0$, $0<\delta<1$, and $\alpha\ge 1$ are given.
	\item A small $\beta_0=\beta_0(d,\rho,\delta,\alpha,\mu)>0$ is chosen. Afterwards, one chooses $0<\beta<\beta_0$.
	\item It will be convenient to introduce an auxiliary parameter $\gamma:=(1+2\mu^{-1})\alpha$. Afterwards, a large $M=M(\beta,\gamma)>0$ is chosen. The parameter $M$ will be used starting from Subsection \ref{subsection_magnitude}.
	\item A small $\ep_0=\ep_0(d,\rho,\delta,M,\beta,\alpha,\mu)>0$ is chosen. The main results will hold for $0<\ep<\ep_0$.	
\end{enumerate}
\end{rem}
In the future, we will use the reference ``$\ep$ is small in the sense of Remark \ref{rem_small_parameters}'' as a shortcut for (1) -- (5), or ``$\beta$ is small in the sense of Remark \ref{rem_small_parameters}'' as a shortcut for (1) and (2).
\begin{rem}
\label{rem_interval_dependence}
In the future, multiple objects will depend on both on $x$ and $x_0$. In most cases, for every object considered at stage $s$, the dependence on $x_0$ is only through the choice of the intervals $I_1,\ldots,I_s$ as in \eqref{eq_u_actual_dependence}, \eqref{eq_interval_specific}:
$$
U^{(s)}(x,x_0)=U^{(s)}(I_1(x_0),I_2(x_0),\ldots,I_s(x_0),x);
$$
in particular, this will also hold for $W^{(s)}(x,x_0)$, $H^{(s)}(x,x_0)$, $f^{(s)}(x,x_0)=H^{(s)}(x,x_0;\bze,\bze)$, $R_{\bn}^{(s)}(x,x_0)$, and $\Lambda_{\bn}^{(s)}(x,x_0)$ (see Lemma \ref{lemma_W_support}).

We should also note that, in principle, one can consider $U^{(s)}(I_1,I_2,\ldots,I_s,x)$ for any collection of intervals, not necessarily of the specific form \eqref{eq_interval_specific} and not necessarily containing $x_0$, assuming that the diagonal dominance condition in (diag1) is satisfied at stages $1,\ldots,s$.
\end{rem}
\begin{thm}
\label{th_convergence}
Let $f$ be $\alpha$-H\"older monotone, $\omega\in \Omega_{\rho,\mu}$, and  $\ep$ is small in the sense of Remark $\ref{rem_small_parameters}$. Then
\begin{itemize}
\item[(conv1)]The matrices $H^{(s)}_{\{\bze,\bm\}}(x,x_0)$, considered in {\rm (diag1)}, are diagonally dominant for $x\in I_{s+1}$.
\item[(conv2)]
The operators $W^{(s)}(x,x_0)$ have finite range:
\bee
\label{eq_finite_range}
W^{(s)}(x,x_0;\bm,\bn)=0,\quad \text{for}\quad |\bm-\bn|\le s(1+\delta/10).
\ene
\item[(conv3)]
The operators $U^{(s)}(x,x_0)$ satisfy the operator norm bound $\|U^{(s)}(x,x_0)-\one\|\le \ep^{s(1-\delta/10)}$.
\item[(conv4)] As a consequence, the limit $W^{(\infty)}(x,x_0)=\lim\limits_{s\to +\infty}W^{(s)}(x,x_0)$ exists in the operator norm topology.
\item[(conv5)] The operators $W^{(\infty)}(x_0,x_0)$ provide a partial diagonalization for $H(x_0)$:
\bee
\label{eq_partial_diagonalization}
\l(\l(W^{(\infty)}(x_0,x_0)\r)^{-1}H(x_0)W^{(\infty)}(x_0,x_0)\r)(\bze,\bn)=0,\quad \text{for}\quad \bn\in \Z^d\setminus\{\bze\}.
\ene
\item[(conv6)]
As a consequence, $W^{(\infty)}(x_0,x_0)e_{\bze}$ is an eigenvector for $H(x_0)$. Denote the corresponding eigenvalue by $E(x_0)$. The function $x_0\mapsto E(x_0)$ is strictly increasing on $[0,1)$. For $\alpha=1$, it is Lipschitz monotone on $[0,1)$.
\end{itemize}
\end{thm}
\begin{rem}
\label{rem_structure_of_proof}
As stated, Property (conv1) is necessary for the operators $W^{(s)}$ to be well-defined, since it relies on a perturbative version of Jacobi rotation. Theorem \ref{th_inductive_estimates} in the next section, which is essentially a quantitative version of (conv1), is the central technical results of the paper. Property (conv2) is mostly a combinatorial fact statement and is proved in the next subsection (see Lemma \ref{lemma_W_support}). Properties (conv3) -- (conv6) are established in Section \ref{section_convergence} as consequences of the inductive estimates.
\end{rem}
\subsection{Structure of the supports of the diagonalizing operators}\label{subsection_supports}
We will now discuss the structure of the supports of $W^{(s)}(x,x_0)$. Let $\rR_{\bn}^{(s)}(x,x_0)$ be defined as in \eqref{eq_R_def}. Define recursively
\begin{multline}
\label{eq_extended_regions}
\Lambda_{\bn}^{(s)}(x,x_0):=\\\begin{cases}
\rR_{\bn}^{(s)}(x,x_0),&s=1;\\
\rR_{\bn}^{(s)}(x,x_0)\cup\l(\bigcup\{\Lambda_{\bm}^{(\ell)}(x,x_0)\colon 1\le \ell\le s-1;\, \Lambda_{\bm}^{(\ell)}(x,x_0)\cap \rR_{\bn}^{(s)}(x,x_0)\neq \varnothing\}\r),&s\ge 2.
\end{cases}	
\end{multline}
We will refer to $\rR_{\bn}^{(s)}(x,x_0)$ as {\it basic regions at step} $s$, and $\Lambda_{\bn}^{(s)}(x,x_0)$ {\it extended regions at step} $s$. An extended region at step $s$ is, by definition, a union of a basic region at step $s$ with all extended regions at previous steps that intersect it. As described above, $U^{(s)}(x,x_0)$ is supported on the basic regions $\{\rR_{\bn}^{(s)}(x,x_0)\colon\bn\in \Z^d,\, x+\bn\cdot\omega\in I_s\}$. Our goal now is to estimate the support of $W^{(s)}(x,x_0)$ in terms of extended regions $\Lambda_{\bn}^{(s)}(x,x_0)$. Namely, we will show that $W^{(s)}(x,x_0)$ is supported on
\bee
\label{eq_W_support}
\{\Lambda_{\bn}^{(\ell)}(x,x_0)\colon 1\le \ell\le s,\,\bn\in \Z^d, \,x+\bn\cdot\omega\in I_\ell, \,\ell\text{ maximal possible} \}.
\ene
Note that every non-empty extended region $\Lambda_{\bn}^{(\ell)}(x,x_0)$ contains extended regions with the same center at all previous steps: that is, $\Lambda_{\bn}^{(\ell-1)}(x,x_0)$, $\Lambda_{\bn}^{(\ell-2)}(x,x_0)$, \ldots. Among those regions, one only considers the largest possible one, in order to avoid (trivial) overlap.
\begin{lem}
\label{lemma_W_support}Assume that $\beta$ is small in the sense of Remark $\ref{rem_small_parameters}$. Then
\begin{enumerate}
	\item Suppose that $\dist(R_{\bn}^{(s)}(x,x_0),R_{\bm}^{(\ell)}(x,x_0))\le 10s$, where $\bn\neq \bm$ and $s\ge \ell$. Then
\bee
\label{eq_intersection_smallregions}
s\ge \log(1/\beta)\ell^{1+\mu/2}.
\ene

\item Suppose that $\dist(R_{\bn}^{(s)}(x,x_0),\Lambda_{\bm}^{(\ell)}(x,x_0))\le 9\ell$, where $\bn\neq \bm$ and $s>\ell$. Then \eqref{eq_intersection_smallregions} also holds.
\item For every $s\ge 1$, we have that $\Lambda_{\bn}^{(s)}(x,x_0)$ is contained in a neighborhood of $R_{\bn}^{(s)}{(x,x_0)}$ of size $10\l(\frac{s}{\log(1/\beta)}\r)^{\frac{1}{1+\mu/2}}$.
\item 
As a consequence,
$$
\diam (\Lambda_{\bn}^{(s)}(x,x_0))\le 2s+20\l(\frac{s}{\log(1/\beta)}\r)^{\frac{1}{1+\mu/2}}\le 3s.
$$
\item The operators $A^{(2)}=U^{(s)}(x,x_0)$ and $A^{(1)}=W^{(s-1)}(x,x_0)$ satisfy the assumptions of Lemma $\ref{lemma_separated}$, with $S_{p}^{(2)}$ being the maximal extended regions and $S_{r}^{(1)}$ being the maximal basic regions. Therefore, $W^{(s)}(x,x_0)$ are supported on $\eqref{eq_W_support}$, for $s=1,2,\ldots$, thus establishing {\rm (conv2)} of Theorem $\ref{th_convergence}$.
\item Recall that the intervals $I_{\ell}$ are chosen as in \eqref{eq_interval_specific}. Suppose that $\tilde I_{\ell}\subset I_{\ell}$ is another collection of intervals. Denote by $\tilde\Lambda_{\bn}^{(s)}(x,x_0)$ the extended regions associated to the intervals $\tilde I_{\ell}$. Then $\tilde\Lambda_{\bn}^{(s)}(x,x_0)\subset \Lambda_{\bn}^{(s)}(x,x_0)$.
\end{enumerate}
\end{lem}
\begin{proof}
Under the assumptions of (1), we have
$$
x+\bn\cdot\omega\in I_{s},\quad x+\bm\cdot\omega\in I_{\ell},\quad |\bn-\bm|\le 12s,
$$
which implies
\bee
\label{eq_s_large}
\|(\bn-\bm)\cdot\omega\|\le 2|I_\ell|=2\beta^{\frac{\ell}{\log(\ell+1)}}.
\ene
Since $\omega\in \Omega_{\rho,\mu}$, we have in view of \eqref{eq_liouville_def}
$$
e^{-\rho|\bn-\bm|^{\frac{1}{1+\mu}}}\le 2\beta^{\frac{\ell}{\ell+1}},
$$
or, equivalently,
$$
|\bn-\bm|\ge \ell^{1+\mu/2}\frac{\ell^{\mu/2}\rho^{-1-\mu} }{\log^{1+\mu}(\ell+1)}\log^{1+\mu}(1/\beta).
$$
Clearly, one can choose a small $\beta$ (in the sense of Remark \ref{rem_small_parameters}) so that
$$
\frac{\ell^{\mu/2}\rho^{-1-\mu} }{\log^{1+\mu}(\ell+1)}\log^{\mu}(1/\beta)\ge 120,\quad\forall\ell\in \N,
$$
which, combined with \eqref{eq_s_large}, implies \eqref{eq_intersection_smallregions}.

For (2) -- (4), we use induction in $s$. Suppose, (2) -- (4) have been established for $0\le s\le s_0$. Put $s=s_0+1$ in (2) and notice that, by the induction assumption (3), $\Lambda_{\bm}^{(\ell)}(x,x_0)$ is contained in a neighborhood of $\rR_{\bm}^{(\ell)}(x,x_0)$ of size $10\l(\frac{\ell}{\log(1/\beta)}\r)^{\frac{1}{1+\mu/2}}\ll \ell$, for $1\le \ell\le s_0$. Therefore, (2) for $s=s_0+1$ follows from (1) and the induction assumption (3) for $0\le s\le s_0$. Using the definition of $\Lambda_{\bn}^{(s)}(x,x_0)$, we have that $\Lambda_{\bn}^{(s_0+1)}(x,x_0)$ must be contained in a $3\ell$-neighborhood of $\rR_{\bn}^{(s_0+1)}(x,x_0)$, where $\ell$ must satisfy $\eqref{eq_intersection_smallregions}$:
$$
\log(1/\beta)\ell^{1+\mu/2}\le s_0+1,
$$
which implies (3) with $s=s_0+1$. Claims (4) and (5) follow from (1) at $\ell=s_0$ and the induction assumption regarding $W^{(s_0)}(x,x_0)$. 

From the definition of $R_{\bn}^{\ell}(x,x_0)$, inclusion of the intervals $\tilde I_\ell\subset I_{\ell}$ implies $\tilde R_{\bn}^{\ell}(x,x_0)\subset R_{\bn}^{\ell}(x,x_0)$, which leads to (6).
\end{proof}
\begin{rem}
\label{remark_variable_intervals}
The definitions of $\Lambda_{\bn}^{(s)}(x,x_0)$ are meaningful for small enough $\beta$ (in the sense of Remark \ref{rem_small_parameters}), since it is important that these regions are finite and disjoint. While their definition is motivated by ``combinatorial'' properties of $H(x)$ (that is, the fact that the operator only involves the nearest neighbor discrete Laplacian), it does not rely on the diagonal dominant properties of the matrices appearing in (diag1) and/or smallness of $\ep$.
\end{rem}
\begin{rem}
\label{rem_extended_boxes}
In the above results (1) -- (3), one can replace extended regions $\Lambda_{\bn}^{(s)}(x,x_0)$ by
\bee
\label{eq_extended_region_variable}
\Lambda_{\bn}^{(s)}(x_0):=\bigcup_{x\in I_s}\Lambda_{\bn}^{(s)}(x,x_0).
\ene
In particular, $\Lambda_{\bze}^{(s)}(x_0)$ is contained in a neighborhood of $\bze$ of size $s(1+\delta/10)$, assuming $\beta$ is small (as usual, in the sense of Remark \ref{rem_small_parameters}).
\end{rem}

\section{Inductive estimates of the matrix elements}
The goal of this section is to prove a quantitative version of (conv1) in Theorem \ref{th_convergence}, that is, the inductive bounds (ind1) -- (ind4) stated in Theorem \ref{th_inductive_estimates} below. These statements require some additional notation and auxiliary estimates which will be introduced in Subsection \ref{subsection_magnitude}.

\subsection{Magnitude of the off-diagonal entries: absorption of small denominators and combinatorial factors}\label{subsection_magnitude}In order to maintain the diagonal dominance property for the matrix elements of $H^{(s)}(x,x_0)$, we will need two ingredients: an upper bound on the numerators in \eqref{eq_delta_definition} and a lower bound on the denominators. In order to measure the numerators, it will be convenient to use the following ``magnitude'' function:
$$
\magn(k):=M^{k-1/2}\beta^{\frac{12\gamma k}{\log(k+1)}},\quad k=1,2,3,\ldots,\quad \gamma:=(1+2\mu^{-1})\alpha.
$$
Here, $M=M(\beta)$ is large. By direct calculation, it is easy to verify the following sub-multiplicative property:
\begin{multline}
\label{eq_submultiplicative1}
\magn(k_1)\magn(k_2)=\\
=M^{-1/2}\beta^{12\gamma k_1\left(\frac{1}{\log(k_1+1)}-\frac{1}{\log(k_1+k_2+1)}\right)}\beta^{12\gamma k_2\left(\frac{1}{\log(k_2+1)}-\frac{1}{\log(k_1+k_2+1)}\right)}\magn(k_1+k_2),
\end{multline}
where both of the exponents are non-negative. Each factor in front on $\magn(k_1+k_2)$ are small.
\begin{rem}
The following elementary bound, which will be useful in Section 3.2, follows directly from the definition.
\bee
\label{eq_elementary_magn}
\magn(r+s)\le M^s\magn(r).
\ene
\end{rem}

In the inductive scheme in the next subsection, the presence of perturbative Jacobi rotations will lead to various expressions of the form
\bee
\label{eq_small_denominator_example}
\frac{|H^{(s)}(\bze,\bk)H^{(s)}(\bk,\bn)|}{f^{(s)}(x,x_0)-f^{(s)}(x+\bk\cdot\omega,x_0)},\quad\text{where}\quad f^{(s)}(x,x_0):=H^{(s)}(x,x_0;\bze,\bze).
\ene
By induction assumptions, we will have 
$$
H^{(s)}(x,x_0;\bm,\bn)\lesssim \ep^{|\bm-\bn|}\magn(|\bm-\bn|),\quad |f^{(s)}(x,x_0)-f^{(s)}(x+\bk\cdot\omega,x_0)|\gtrsim\|\bk\cdot\omega\|^{\alpha},
$$
see Theorem \ref{th_inductive_estimates} and Definition \ref{def_order} for more precise statements. The factors in front of $\magn(k_1+k_2)$ in \eqref{eq_submultiplicative1} will be used to absorb the denominator in \eqref{eq_small_denominator_example} and, ultimately, estimate the expressions of the form \eqref{eq_small_denominator_example} by $\ep^{|\bn|}\magn(|\bn|)$, under some assumptions on $\bk$ and $\bn$. 

Define
\bee
\label{eq_level_definition}
\level(\bn):=\max\{\ell\colon 2\|\bn\cdot\omega\|\le \beta^{(\ell)}\},
\ene
In other words, recalling the definition \eqref{eq_betas_def} of $\beta^{(\ell)}$,
\bee
\label{eq_level_equivalent}
\level(\bn)=\ell\quad\text{if and only if}\quad \frac12 \beta^{\frac{\ell+1}{\log(\ell+2)}}<\|\bn\cdot\omega\|\le \frac12 \beta^{\frac{\ell}{\log(\ell+1)}}.
\ene
\begin{lem}
\label{lemma_generalized_absorption}
Define $\magn(\cdot)$ as above, and suppose that
$$
|\bn|=k_2\ge k_1\ge \level(\bn).
$$
Assume that $\beta$ is small in the sense of Remark $\ref{rem_small_parameters}$. Then
\bee
\label{eq_generalized_absorption}
\frac{\magn(k_1)\magn(k_2)}{\|\bn\cdot\omega\|^{\alpha}}<\frac{\magn(k_1)\magn(k_2)}{\|\bn\cdot\omega\|^{2\alpha}}\le M^{-1/2}\beta^{\frac{\gamma k_1}{\log^2 (k_1+1)}}\magn(k_1+k_2).
\ene
\end{lem}
\begin{proof}
The first inequality is trivially true (for $\omega$ with rationally independent components), it is included for convenience, since the estimates with both kinds of denominators will appear in the future.

Suppose that
\bee
\label{eq_k1k2_constraints}
0\neq|\bn|=k_2\ge k_1\ge \ell:=\level(\bn).
\ene
Recall that $\omega\in\Omega_{\rho,\mu}$: that is,
$$
\|\bn\cdot\omega\|\ge e^{-\rho|\bn|^{\frac{1}{1+\mu}}},\quad \forall\bn\in \Z^d\setminus\{\bze\}.
$$
As a consequence, assuming $\beta$ is small (similarly to the proof of Lemma \ref{lemma_W_support}),
\bee
\label{eq_level_distance}
k_2=|\bn|\ge \rho^{-(1+\mu)}\l(\frac{\ell}{\log(\ell+1)}\r)^{1+\mu}\log^{1+\mu}(1/\beta)\ge \ell^{1+\mu/2}\log(1/\beta).
\ene
Let
$$
m(k_1,k_2):=\frac{k_1}{\log(k_1+1)}+\frac{k_2}{\log(k_2+1)}-\frac{ (k_1+k_2)}{\log(k_1+k_2+1)},
$$
so that \eqref{eq_submultiplicative1} can be rewritten as
$$
\magn(k_1)\magn(k_2)=M^{-1/2}\beta^{12\gamma m(k_1,k_2)}\magn(k_1+k_2).
$$
One can easily check that $m$ is strictly increasing in both of its arguments:
\bee
\label{eq_monotonicity_m}
m(k_1+a,k_2)>m(k_1,k_2),\quad m(k_1,k_2+a)>m(k_1,k_2),\quad \forall a>0.
\ene
Note that, for $k_2\ge k_1$, one can estimate the expression in the exponent of the right hand side of \eqref{eq_submultiplicative1} by
\bee
\label{eq_m_first_estimate}
\frac{1}{\log(k_1+1)}-\frac{1}{\log(k_1+k_2+1)}\ge \frac{1}{10\log^2(k_1+1)}.
\ene
In view of \eqref{eq_level_distance}, since $k_2\gg \ell^{1+\mu/2}$, we also have
\bee
\label{eq_m_second_estimate}
\frac{1}{\log(\ell+1)}-\frac{1}{\log(k_2+\ell+1)}\ge \l(\frac{\mu}{2+\mu}\r)\frac{1}{\log(\ell+1)}
\ene
Using \eqref{eq_monotonicity_m} and \eqref{eq_submultiplicative1} (in particular, $12m(k_1,k_2)\ge 10m(k_1,k_2)+2m(\ell,k_2)$), we obtain
\begin{multline*}
\magn(k_1)\magn(k_2)\le M^{-1/2}\beta^{10 \gamma m(k_1,k_2)}\beta^{2\gamma m(\ell,k_2)}\magn(k_1+k_2)\\ \le 
M^{-1/2}\beta^{\frac{\gamma k_1}{\log^2 (k_1+1)}}\beta^{\l(\frac{\mu}{2+\mu}\r)\frac{2\gamma \ell}{\log(\ell+1)}}\magn(k_1+k_2)\le 
M^{-1/2}\beta^{\frac{\gamma k_1}{\log^2 (k_1+1)}}\beta^{\frac{2\alpha \ell}{\log(\ell+1)}}\magn(k_1+k_2).
\end{multline*}
Finally, \eqref{eq_level_equivalent} implies
$$
\|\bn\cdot\omega\|^{2\alpha}\le \frac14\beta^{\frac{2\alpha \ell}{\log(\ell+1)}},
$$
which completes the proof.
\end{proof}
\begin{rem}
\label{rem_worst_level}
If $k_1'\leq k_1$ and $\ep$ is small in the sense of Remark \ref{rem_small_parameters}, then the following inequality easily follows from the definition of $\magn(\cdot)$:
\bee
\label{eq_worst_level}
\ep^{k_1+k_2}M^{-1/2}\beta^{\frac{\gamma k_1}{\log^2 (k_1+1)}}\magn(k_1+k_2)\le \ep^{k_1'+k_2} M^{-1/2}\beta^{\frac{\gamma k_1'}{\log^2 (k_1'+1)}}\magn(k_1'+k_2).
\ene
The presence of small $\ep$ in an appropriate power is important here, since $M$ is chosen to be large after choosing a small $\beta$. In view of Lemma \ref{lemma_generalized_absorption}, we can understand it in the following sense. Suppose that $\bn,\bn',k_1,k_2,k_1',\ell,\ell'$ are chosen as follows:
\bee
\label{eq_worst_denominator_assumptions}
\begin{gathered}
	|\bn|=k_2\ge k_1\ge \ell:=\level(\bn);\\
|\bn'|=k_2\ge k_1'\ge \ell':=\level(\bn');\\
k_1'\le k_1;\quad \ell'\le \ell.
\end{gathered}
\ene
Then, from Lemma \ref{lemma_generalized_absorption}, we have
\bee
\label{eq_test_denominator_1}
\frac{\magn(k_1)\magn(k_2)}{\|\bn\cdot\omega\|^{2\alpha}}\le  \ep^{k_1+k_2}M^{-1/2}\beta^{\frac{\gamma k_1}{\log^2 (k_1+1)}}\magn(k_1+k_2),
\ene
\bee
\label{eq_test_denominator_2}
\frac{\magn(k_1')\magn(k_2)}{\|\bn'\cdot\omega\|^{2\alpha}}\le  \ep^{k_1'+k_2}M^{-1/2}\beta^{\frac{\gamma k_1'}{\log^2 (k_1'+1)}}\magn(k_1'+k_2).
\ene
Due to \eqref{eq_worst_level}, both of these expressions can be estimated by the right hand side of \eqref{eq_test_denominator_2}. In other words, in comparing two expressions of the kind estimated by Lemma \ref{lemma_generalized_absorption} with the constraints \eqref{eq_worst_denominator_assumptions}, the ``worst`` case corresponds to the denominator of smaller level. This will help to reduce the number of cases later in the proof of Theorem \ref{th_inductive_estimates}. 
\end{rem}
\subsection{The inductive statements}\label{subsection_inductive} 
\begin{defn}
\label{def_order}
We will say that a matrix element $H^{(s)}(x,x_0;\bm,\bn)$ is {\it of order $r\geq1$} if the following estimate is true:
\bee
\label{eq_order_definition}
|H^{(s)}(x,x_0;\bm,\bn)|\le \begin{cases}
\frac{s+1}{s+2}\,\magn(r)\ep^r,&|\bm-\bn|\le s;\\
\frac{1}{|\bm-\bn|-s}\,\magn(r)\ep^r,&s+1\le |\bm-\bn|\le s(1+\delta/10).
\end{cases}
\ene	
\end{defn}
\begin{rem}
Note that the order of a matrix element is not uniquely defined (an element of order $r$ will also be an element of any smaller order, assuming $\ep\ll\beta$). As a consequence, this definition is merely a convenient way to state the inequality \eqref{eq_order_definition}. To emphasize that, we will sometimes use the words ``a matrix element is of order at least $r$'', with the same meaning. The conclusion of Lemma \ref{lemma_W_support} implies that $H^{(s)}(x,x_0;\bm,\bn)=0$ for $|\bm-\bn|>s(1+\delta/10)$, which is why that case is not considered in \eqref{eq_order_definition}.
\end{rem}
\begin{thm}
\label{th_inductive_estimates}
For $s\in \N$, define $H^{(s)}(x,x_0)$ as in Subsection $\ref{subsection_construction}$, with $f^{(s)}(x,x_0):=H^{(s)}(x,x_0;\bze,\bze)$. Assume that the parameters are chosen according to Remark $\ref{rem_small_parameters}$.
 Then
\begin{itemize}
\item[(ind1)] The order of $H^{(s)}(x,x_0;\bm,\bn)$ is (at least) $|\bm-\bn|$.
\item[(ind2)] Suppose that $x+\bm\cdot\omega\in I_k$ or $x+\bn\cdot\omega\in I_k$, $k\le s$, $\bm\neq \bn$. Then the order of $H^{(s)}(x,x_0;\bm,\bn)$ is (at least) $k+1$.


\item[(ind3)] 
We have
$$
|f^{(s)}(x,x_0)-f^{(s-1)}(x,x_0)|\le \ep^{2s-1} \magn(2s-1),\quad \forall x\in \R.
$$
\item[(ind4)] 
For $x\in I_s$, $y\in I_k$ or $x\in I_k$, $y\in I_s$, $0\le k\le s$, $0\le x<y<1$, we have
$$
f^{(s)}(y,x_0)-f^{(s)}(x,x_0)\ge 2^{1-\alpha}(y-x)^{\alpha}-4^d\l(1-\frac{1}{s-k+2}\r)\ep^{k}\magn(k).
$$
\end{itemize}
\end{thm}
\begin{rem}
\label{rem_lipschitz}
It is easy to see that (ind4) on scale $s-1$ together with (ind1) -- (ind3) on the scale $s$ implies the following:
\bee
\label{eq_lower_lipschitz}
\frac{f^{(s)}(x,x_0)-f^{(s)}(x+\bn\cdot\omega,x_0)}{\mathrm{sgn}(\{x\}-\{x+\bn\cdot\omega\})\l|\{x\}-\{x+\bn\cdot\omega\}\r|^{\alpha}}\ge\frac{1}{2^{\alpha}},\quad \text{for}\quad x\in I_s,\quad \|\bn\cdot\omega\|\ge 10|I_s|.
\ene
In particular, one can always choose $\beta$ (within Remark \ref{rem_small_parameters}) so that the above estimate holds for $0<|\bn|\le 10s$.
\end{rem}
\subsection{Proof of (ind1) -- (ind3): general structure of the terms.}\label{subsection_general}Recall that the structure of $U^{(s+1)}(x,x_0)$, defined in (diag1) -- (diag4) of Subsection \ref{subsection_construction}, involves a sequence of perturbative Jacobi rotations applied to various $(2\times 2)$-blocks of $H^{(s)}(x,x_0)$. In this section, we will show that, for practical purposes, each perturbative Jacobi rotation matrix \eqref{eq_u_structure} can be replaced by the leading term in its asymptotic expansion. Let
$$
r_{\bk}(x)=r_{\bk}^{(s)}(x,x_0):=\frac{H^{(s)}(x,x_0;\bze,\bk)}{f^{(s)}(x,x_0)-f^{(s)}(x+\bk\cdot\omega,x_0)}.
$$
For simplicity of the notation, we will drop the dependence on $s$ and $x_0$ in the notation for $r_{\bk}(x)$. Recall that
$$
U^{(s+1),\bk}(x,x_0)=\begin{pmatrix}
	1&r_{\bk}(x)\\
	-r_{\bk}(x)&1
	\end{pmatrix}+\l(r_{\bk}(x)\r)^2W(r_{\bk}(x)),
$$
where $W(\cdot)$ is a real analytic function in its argument bounded by $10$ in a neighborhood of the origin. Let
$$
M_{\bk}\colon \ell^2(\Z^d)\to \ell^2(\Z^d),\quad M_{\bk}e_{\bn}:=\begin{cases}
-e_{\bk},&\bn=\bze;\\
e_{\bze},&\bn=\bk;\\
0,&\bn\in \Z^d\setminus\{\bze,\bk\}.
\end{cases}
$$
In other words, $M_{\bk}$ acts as $\begin{pmatrix}
0&1\\-1&0
\end{pmatrix}$ in the subspace spanned by $\{e_{\bze},e_{\bk}\}$, and by zero everywhere else on $\ell^2(\Z^d)$.
Then
\bee
\label{eq_expansion_1}
U^{(s+1),\bk}(x,x_0)=I+r_{\bk}(x)M_{\bk}+r_{\bk}^2(x) W(r_{\bk}(x)),
\ene
and
\bee
\label{eq_expansion_2}
(U^{(s+1),\bk}(x,x_0))^{-1}=(U^{(s+1),\bk}(x,x_0))^{T}=I-r_{\bk}(x)M_{\bk}+r_{\bk}^2(x) W(r_{\bk}(x))^T.
\ene
Recall that
$$
U^{(s+1)}_0(x,x_0)=\prod_{0<|\bk|\le s+1}U^{(s+1),\bk}(x,x_0),
$$
\begin{multline}
\label{eq_h_expression}	
H^{(s+1)}(x,x_0;\bm,\bn)=\sum\limits_{\br,\bl}U^{(s+1)}_0(x,x_0;\br,\bm)H^{(s)}(x,x_0;\br,\bl)U^{(s+1)}_0(x,x_0;\bl,\bn),\\ 0\le |\bm|\le 3s;\,\,0\le |\bn|\le 3s;\,\,x\in I_s.
\end{multline}
Note that the full expression for $H^{(s+1)}(x,x_0)$ involves 
\bee
\label{eq_quasiperiodic_product}
U^{(s+1)}(x,x_0):=\prod_{\bm\in \Z^d}T^{-\bm}U_0^{(s+1)}(x+\omega\cdot\bm,x_0)T^{\bm}
\ene
rather than the individual factor $U^{(s+1)}_0(x,x_0)$. However, for simplicity of calculations and without loss of generality, we restrict ourselves to a $3s$-neighborhood of the origin in $\Z^d$, in which case one only needs to consider $\bm=0$ in \eqref{eq_quasiperiodic_product}. In this case, assuming that $x\in I_s$ and in view of Lemma \ref{lemma_W_support}, the summation in \eqref{eq_h_expression} can be narrowed down to the range
$$
|\br|,|\bl|\le s(1+\delta/10).
$$
Estimates in the remaining region can be recovered using \eqref{eq_quasiperiodic_product}.

It will be convenient to introduce some additional structure in the inductive estimates of the off-diagonal entries of $H^{(s+1)}(x,x_0)$. Clearly, \eqref{eq_h_expression} can be interpreted as a replacement of every matrix element of $H^{(s)}(x,x_0,\bm,\bn)$ by a new expression. As $s$ changes to $s+1$, both (ind1) and (ind2) impose some conditions on the order of the matrix element $H^{(s+1)}(x,x_0,\bm,\bn)$. There are two scenarios: either the declared order of a matrix entry is determined by (ind1), in which case it stays the same (adjustment scenario), or it is determined by (ind2), in which case the order increases by $1$ (replacement scenario).
\begin{itemize}
	\item {\it Adjustment scenario.} In this case, the estimates will use the ``wiggle room'' in the definition \eqref{eq_order_definition} as $s$ changes to $s+1$. More precisely, we will show that
\begin{multline}
\label{eq_wiggle_room}
|H^{(s+1)}(x,x_0;\bm,\bn)-H^{(s)}(x,x_0;\bm,\bn)|\le \\\le \begin{cases}
\l(\frac{s+2}{s+3}-\frac{s+1}{s+2}\r)\magn(r)\ep^r,&|\bm-\bn|\le s+1;\\
\l(\frac{s+2}{s+3}-\frac12\r)\magn(r)\ep^r,&|\bm-\bn|=s+2;\\
\l(\frac{1}{|\bm-\bn|-(s+1)}-\frac{1}{|\bm-\bn|-s}\r)\magn(r)\ep^r,&s+3\le |\bm-\bn|\le (s+1)(1+\delta/10).
\end{cases}
\end{multline}
If $H^{(s)}(x,x_0;\bm,\bn)$ was of order $r$, the estimates \eqref{eq_wiggle_room} will ensure that it will remain of order $r$. Note that the case $|\bm-\bn|=s+2$ corresponds to the borderline situation in \eqref{eq_order_definition}, where the expression used in the definition of the order changes as $s$ changes to $s+1$.
	\item {\it Replacement scenario}: the declared order of the matrix element has increased (it is easy to see that the increase can only be by $1$). The improvement comes from the fact that the term $H^{(s)}(x,x_0;\bm,\bn)$ appears twice in the leading term expansion of \eqref{eq_h_expression}, with opposite signs, and cancels. In this case, the total contribution of the remaining terms of \eqref{eq_h_expression} is estimated by the corresponding value in \eqref{eq_order_definition}. It is easy to see that the latter will always be $\frac{s+2}{s+3}\magn(s+2)\ep^{s+2}$, since (ind2) at stage $s+1$ can only impose the order of a matrix element to be equal to $s+2$. We will often obtain a better estimate (say, $\frac{s+1}{s+2}\magn(s+2)\ep^{s+2}$) and use the remaining wiggle room \eqref{eq_wiggle_room} to absorb some less important terms.
\end{itemize}

\subsection{Proof of (ind1) -- (ind3): the quadratic terms}\label{subsection_quadratic}
Let us expand \eqref{eq_h_expression} using \eqref{eq_expansion_1} and \eqref{eq_expansion_2}. A term of this expansion will be called {\it quadratic} if it contains a factor of the form $r_{\bm}(x)r_{\bn}(x)$ (possibly with $\bm=\bn$).
\begin{thm}Suppose that $H^{(s)}(x,x_0;\bm,\bn)$ is of order $r$, $\bm\neq \bn$. Consider the sum of absolute values of all quadratic terms that contribute to that matrix element. Denote by $Q^{r,(s)}(x,x_0)$ the maximal possible value of such sums over all matrix elements of order $r$. Then
$$
|Q^{r,(s)}(x,x_0)|\le \frac{1}{s^{10}}\magn(r+s)\ep^{r+s}.
$$
As a consequence, at the induction step, the contribution of all quadratic terms can be absorbed into \eqref{eq_wiggle_room}.
\end{thm}
\begin{proof}

If both $\bn,\ \bm\not=\bze$, then the quadratic contributions to $ H^{(s+1)}(x,x_0;\bm,\bn)$ come from the cases 1) $\br=\bm$, $\bl=\bn$, 2) $\br\not=\bm$, $\bl=\bn$, 3) $\br=\bm$, $\bl\not=\bn$, 4) $\br\not=\bm$, $\bl\not=\bn$. It follows from the structure of the matrices $U^{(s)}(x,x_0)$ that the corresponding contributions can be estimated by, respectively,

$$10|H^{(s)}(x,x_0;\bm,\bn)|(|r_{\bm}(x)|^2+|r_{\bn}(x)|^2)\prod_{0<|\bk|<s+1}(1+2|r_{\bk}(x)|)=:J_1,$$ 
$$10|H^{(s)}(x,x_0;\br,\bn)||r_{\bm}(x)||r_{\br}(x)|\prod_{0<|\bk|<s+1}(1+2|r_{\bk}(x)|)=:J_2,$$ 
$$10|H^{(s)}(x,x_0;\bm,\bl)||r_{\bn}(x)||r_{\bl}(x)|\prod_{0<|\bk|<s+1}(1+2|r_{\bk}(x)|)=:J_3,$$
$$10|H^{(s)}(x,x_0,\br,\bl)||r_{\bm}(x)||r_{\br}(x)||r_{\bn}(x)||r_{\bl}(x)|\prod_{0<|\bk|<s+1}(1+2|r_{\bk}(x)|)=:J_4.$$ 

The cases $\bm\not=\bn=\bze$ and $\bn\not=\bm=\bze$ are equivalent, so we consider the former. Again, the contributions come from the cases above with $\bn=\bze$ and are estimated as

$$10|H^{(s)}(x,x_0;\bm,\bze)|(|r_{\bm}(x)|^2+\sum_{0<|\bk|<s+1}|r_{\bk}(x)|^2)\prod_{0<|\bk|<s+1}(1+2|r_{\bk}(x)|)=:J_1',$$ 
$$10|H^{(s)}(x,x_0;\br,\bze)||r_{\bm}(x)||r_{\br}(x)|\prod_{0<|\bk|<s+1}(1+2|r_{\bk}(x)|)=:J_2',$$ 
$$10|H^{(s)}(x,x_0;\bm,\bl)||r_{\bl}(x)|^2\prod_{0<|\bk|<s+1}(1+2|r_{\bk}(x)|)=:J_3',$$ 
$$10|H^{(s)}(x,x_0;\br,\bl)||r_{\bm}(x)||r_{\br}(x)||r_{\bl}(x)|\prod_{0<|\bk|<s+1}(1+2|r_{\bk}(x)|)=:J_4'.$$ 

First, notice that by the induction assumption we have $|r_{\bk}(x)|<\varepsilon^{|\bk|/2}$ (for $\varepsilon$ small enough in the sense of Remark \ref{rem_small_parameters}), and thus $\prod_{0<|\bk|<s+1}(1+2|r_{\bk}(x)|)\leq2$. Next, we use \eqref{eq_generalized_absorption} and induction assumptions to estimate
$$
|r_{\bk_1}(x)r_{\bk_2}(x)|\leq \varepsilon^{p_1+p_2}M^{-1/2}\beta^{\frac{\gamma s}{\log^2(s+1)}}\magn(p_1+p_2),
$$
where $p_j=\max\{|\bk_j|,s\}$. Now, the estimates for all $J_j, J_j'$ follow after one (or several, depending on the case) more applications of  \eqref{eq_generalized_absorption}. It is easy to see that the corresponding estimates match or supersede the powers of $\varepsilon$ required in (ind1) or (ind2) for $s+1$. It is also important to note that factor $M^{-1/2}\beta^{\frac{\gamma s}{\log^2(s+1)}}$ not only takes care of the coefficients in \eqref{eq_order_definition} but also suppresses the combinatorial factor coming from the total number of terms (estimated by $(6s)^{2d}$).
\end{proof}

\subsection{Proof of (ind1) -- (ind3): non-quadratic terms.}\label{subsection_non_quadratic}Throughout the remaining part of the proof, we assume that $x\in I_{s+1}$. The expression for $H^{(s+1)}(x,x_0)$ can be rewritten as follows:
$$
H^{(s+1)}(x,x_0)=H^{(s)}(x,x_0)-\sum_{0<|\bk|\le s+1}r_{\bk}(x)[M_{\bk},H^{(s)}(x,x_0)]+\textrm{quadratic terms},
$$
where again it is only considered for the matrix elements with $0\le |\bm|,|\bn|\le s(1+\delta/10)$.

From the definition of $r_{\bk}(x)$, we have the following for $0<|\bk|\le s+1$:
\bee
\begin{split}
\label{eq_cancellation}
H^{(s)}(x,x_0;\bze,\bk)=r_{\bk}(x)[M_{\bk},H^{(s)}(x,x_0)](\bze,\bk);\\
H^{(s)}(x,x_0;\bk,\bze)=r_{\bk}(x)[M_{\bk},H^{(s)}(x,x_0)](\bk,\bze).
\end{split}
\ene
As a consequence, the original entry $H^{(s)}(x,x_0;\bze,\bk)$ is cancelled by the corresponding entry of the commutator term with the same $\bk$.
Let us now consider the non-quadratic terms contributing to $H^{(s+1)}(x,x_0;\bm,\bn)$. They can be of the form
$$
r_{\bk}(x)M_{\bk}(\bm,\br)H^{(s)}(x,x_0;\br,\bn),\quad \textrm {or}\quad -r_{\bk}H^{(s)}(x,x_0;\bm,\br)M_{\bk}(\br,\bn).
$$
Since $H^{(s+1)}(x,x_0)$ is symmetric and the inductive statements are symmetric with respect to interchanging of $\bm$ and $\bn$, it is sufficient to consider the terms of the first kind. In these terms the structure of $M_{\bk}$ implies that the only two possible choices are $\bm=\bze$, $\br=\bk$ or $\bm=\bk$, $\br=\bze$ (with the corresponding entry being equal to $\pm 1$). As a consequence, it is sufficient to consider the following four types of contributions (recall that we always have $0<|\bk|\le s+1$):
\begin{enumerate}
	\item[(C1)] $-r_{\bk}H^{(s)}(x,x_0;\bze,\bn)$ to $H^{(s+1)}(x,x_0;\bk,\bn)$, $\bk\neq \bn$. In this case, we have only one term for each matrix element of $H^{(s+1)}(x,x_0)$ (adjustment scenario).
	\item[(C2)] $r_{\bk}H^{(s)}(x,x_0;\bk,\bn)$ to $H^{(s+1)}(x,x_0;\bze,\bn)$, $|\bn|\le s+1$. These terms are indexed by $\bk$ with $0<|\bk|\le s+1$, $\bk\neq \bn$. The term with $\bk=\bn$ cancels  the original term $H^{(s)}(x,x_0;\bze,\bn)$ and is therefore not considered (replacement scenario).
	\item[(C3)] $r_{\bk}(x)H^{(s)}(x,x_0;\bk,\bn)$ to $H^{(s+1)}(x,x_0;\bze,\bn)$, $|\bn|\ge s+2$. Similarly to the previous case, these terms are also indexed by $\bk$ with $0<|\bk|\le s+1$. However, this range no longer includes $\bk=\bn$, and therefore no cancellation is taking place (adjustment scenario).
	\item[(C4)]Contributions to the diagonal terms: $-r_{\bk}H^{(s)}(x,x_0;\bze,\bk)$ to $H^{(s+1)}(x,x_0;\bk,\bk)$ or $r_{\bk}H^{(s)}(x,x_0;\bk,\bze)$ to $H^{(s+1)}(x,x_0;\bze,\bze)$. In both cases, this will affect (ind3).
\end{enumerate}
\vskip 2mm
{\bf Type (C1)}. For every $\bk$ under consideration, estimate \eqref{eq_level_distance} (see also the proof of Lemma \ref{lemma_W_support} (1)) implies
\bee
\label{eq_levelk}
\level(\bk)\le \l(\frac{s}{\log(1/\beta)}\r)^{\frac{1}{1+\mu/2}}\ll s,
\ene
assuming that $\beta$ is small in the sense of Remark \ref{rem_small_parameters}. Let
$$
p:=\max\{|\bk|,s+1\},\quad  q:=\max\{|\bn|,s+1\}.
$$
The entries $H(x,x_0;\bze,\bk)$ and $H(x,x_0;\bze,\bn)$ are of order $p$, $q$, respectively. Since both are at least $s+1$ and in view of \eqref{eq_levelk}, one can estimate the contribution (C1) 
using Lemma \ref{lemma_generalized_absorption} by
\begin{multline}
\label{eq_case1_type1}
|r_{\bk}H^{(s)}(x,x_0;\bze,\bn)|\le \frac{|H^{(s)}(x,x_0;\bze,\bn)H^{(s)}(x,x_0;\bze,\bk)|}{f^{(s)}(x,x_0)-f^{(s)}(x+\bk\cdot\omega,x_0)}\le 2\frac{\ep^{p+q}\magn(p)\magn(q)}{\|\bk\cdot\omega\|^{\alpha}}\le \\ \le \ep^{p+q}M^{-1/2}\beta^{\frac{\gamma s}{\log^2 (s+1)}}\magn(p+q)\ll \ep^{p+q}\frac{1}{(s+2)^{10}}\magn(p+q).
\end{multline}
This can be absorbed into \eqref{eq_wiggle_room}.
\vskip 5mm
{\noindent \bf Type (C2)}. Recall that $x\in I_{s+1}$. As mentioned above, in this case cancellation \eqref{eq_cancellation} eliminates the term $H^{(s)}(x,x_0;\bze,\bn)$.  From the definition of $\level(\bk)$, it is easy to see that $x+\bk\cdot\omega\in I_{\level(\bk)}$, assuming that $\level(\bk)\le s$ (which will always be true provided $|\bk|\le 10s$). As a consequence, every entry $H^{(s)}(x,x_0;\bk,\bn)$ is of the order at least $\max\{|\bn-\bk|,\level(\bk)+1\}$. Recall also that, by induction assumption, the order of $H^{(s)}(x,x_0;\bze,\bk)$ is at least $s+1\gg \level(\bk)$. This justifies using Lemma \ref{lemma_generalized_absorption}. In view of Remark \ref{rem_worst_level}, one can without loss of generality assume $\level(\bk)=0$. As a consequence, the contribution of the terms of type  (C2) with $0<|\bn|\le s+1$ can be estimated by
\begin{multline*}
\sum\limits_{\bk\colon 0<|\bk|\le s+1;\,\bk\neq \bn}|r_{\bk}H^{(s)}(x,x_0;\bk,\bn)|\le \frac{|H^{(s)}(x,x_0;\bze,\bk)H^{(s)}(x,x_0;\bk,\bn)|}{f^{(s)}(x,x_0)-f^{(s)}(x+\bk\cdot\omega,x_0)}\\ \le
\sum\limits_{\bk\colon 0<|\bk|\le s+1;\,\bk\neq \bn}\ep^{s+1+|\bk-\bn|}M^{-1/2}\beta^{\frac{\alpha\min\{|\bk-\bn|,s+1\}}{\log^2 (\min\{|\bk-\bn|,s+1\}+1)}}\magn(s+1+
|\bk-\bn|)\\
\le \sum_{r=0}^{s}(2r+1)^d\,\ep^{s+r+2}M^{-1/2}\beta^{\frac{\gamma r}{\log^2(r+1)}}\magn(s+r+2)\\+ \sum_{r=s+1}^{3s}(2r+1)^d\,\ep^{s+r+2}M^{-1/2}\beta^{\frac{\gamma (s+1)}{\log^2(s+2)}}\magn(s+r+2)\\ \le 
2\ep^{s+2}\magn(s+2) M^{-1/2} \sum_{r=0}^{s}(2r+1)^d\,\ep^{r}\beta^{\frac{\gamma r}{\log^2(r+1)}}M^r\le \frac12 \ep^{s+2}\magn(s+2).
\end{multline*}
The first inequality follows from Lemma \ref{lemma_generalized_absorption}, in view of Remark \ref{rem_worst_level}. In the second inequality, we assumed $|\bk-\bn|=r+1$, and the factor $(2r+1)^d$ is an estimate on the number of possible values of $\bk$ for a given $r$. It is easy to see that the second sum is dominated by the first sum assuming the smallness of the parameters. Afterwards, we applied \eqref{eq_elementary_magn} 
$$
\magn(r+s)\le M^s\magn(r)
$$
to convert $\magn(s+r+2)$ into $\magn(s+2)$, which is placed outside of the summation together with $\ep^{s+2}$. It is clear that last sum, together with the factor $2$ in front, can be absorbed into $M^{-1/2}$, since the choice of $\ep$ comes after $M$ and $\beta$.

\vskip 5mm

{\noindent  \bf Type (C3)}. In this case, one needs to estimate the same expression as in the previous case. Since $|\bk|\le s+1$, $|\bn|\ge s+2$, we no longer have a term with $\bk=\bn$, and therefore no cancellation of the kind \eqref{eq_cancellation} takes place. In other words, we are dealing with the adjustment scenario, and the bound of the kind $\frac12 \ep^{s+2}\magn(s+2)$ on the contribution of Type (C2) is no longer sufficient. Nevertheless, the estimates can be started in the same way. In view of Remark \ref{rem_worst_level}, assume $\level(\bk)=0$. We have
\begin{multline}
\label{eq_sum_case22}
\sum\limits_{\bk\colon 0<|\bk|\le s+1}|r_{\bk}H^{(s)}(x,x_0;\bk,\bn)|\le \\ 
\le \sum\limits_{\bk\colon 0<|\bk|\le s+1}\ep^{\max\{s+1,|\bk|\}+|\bk-\bn|}M^{-1/2}\beta^{\frac{\gamma\min\{s+1,|\bk-\bn|\}}{\log^2 (\min\{s+1,|\bk-\bn|\}+1)}}\magn(\max\{s+1,|\bk|\}+|\bk-\bn|).
\end{multline}
The lowest possible value of $\max\{s+1,|\bk|\}+|\bk-\bn|$ is $|\bn|$, which matches the declared order of this contribution. It remains to show that the powers of $\beta$ and the factor $M^{-1/2}$, together with the combinatorial factors and lower order terms, can be absorbed into the pre-factors in the second and third cases of \eqref{eq_wiggle_room}.

Similarly to the calculation for Type (C2), denote 
$$
r:=|\bk-\bn|-|\bn|+\max\{s+1,|\bk|\}.
$$
Clearly, we have $0\le r\le 3s$. For a fixed value of $r$, there are at most $(2(|\bn|-s-2+r)+1)^d$ choices of the value of $\bk$ in the summation \eqref{eq_sum_case22}. Therefore, we can continue the estimates
\begin{multline*}
\sum\limits_{\bk\colon 0<|\bk|\le s+1}|r_{\bk}H^{(s)}(x,x_0;\bk,\bn)|\le\\ 
\le \ep^{|\bn|}M^{-1/2}\sum_{0\le r\le 3s}(2(|\bn|-s-2+r)+1)^d  \ep^r \beta^{\frac{ \gamma(\min\{s+1,|\bn|-s-2+r\})}{\log^2 (\min\{s+1,|\bn|-s-2+r\}+1)}}\magn(|\bn|+r)\\
\le \ep^{|\bn|}\magn(|\bn|)M^{-1/2}\sum_{0\le r\le 3s}(2(|\bn|-s-2+r)+1)^d \ep^r  \beta^{\frac{ \gamma(\min\{s+1,|\bn|-s-2+r\})}{\log^2(\min\{s+1,|\bn|-s-2+r\}+1)}}M^r\\
\le\ep^{|\bn|}\magn(|\bn|)M^{-1/2}\sum_{t=|\bn|-s-2}^{4s}(2t+1)^d \beta^{\frac{\gamma \min\{t,s+1\}}{\log^2 (\min\{t,s+1\}+1)}}(\ep M)^{t-(|\bn|-s-2)}\\
\le\ep^{|\bn|}\magn(|\bn|)M^{-1/2}\beta^{\frac{1}{10}\gamma(|\bn|-s-2)^{1/2}},
\end{multline*}
assuming that $\ep M\ll 1$. Clearly, if $M$ is large enough, one can guarantee
$$
M^{-1/2}\beta^{\frac{1}{10}\gamma(|\bn|-s-2)^{1/2}}\ll \begin{cases}
\frac{s+2}{s+3}-\frac12,&|\bn|=s+2;\\
\frac{1}{|\bn|-s-1}-\frac{1}{|\bn|-s},& |\bn|\ge s+3,
\end{cases}
$$
which matches the estimate in \eqref{eq_wiggle_room} (recall that $\ep$ is chosen after $M$).

{\noindent  \bf Type (C4)}. These estimates are similar to those of type (C1). In both cases, the expressions produce an estimate of the kind $\ep^{2s}\magn(2s)$ (or less), with a factor $M^{-1/2}\beta^{\frac{\gamma s}{\log^2 (s+1)}}$ sufficient to absorb any combinatorial factor. This also applies to the contribution of quadratic terms.

\section{Monotonicity of finite volume eigenvalues and the proof of (ind4)}\label{subsection_monotonicity}
While the diagonal entries of the original operator $H(x)=H^{(0)}(x)$ are monotone on suitably chosen intervals, the partial diagonalization procedure would, in general, destroy exact monotonicity and, at the very least, would not provide any means of controlling it directly. Approximate monotonicity, such as in (ind4), appears to be satisfied on the first several scales, but the perturbation estimates alone do not seem to provide any insight why the accuracy of this property on smaller intervals would improve as $s$ becomes larger. It seems that the solution to this issue should rely on some properties of the operator that are external in relation to the partial diagonalization procedure.

Our approach to addressing this issue is based on the following observation. Due to Lemma \ref{lemma_W_support}, the partial diagonalization transformation is a finite range operator, and therefore, for a given value of $s$, it does not distinguish between infinite volume and finite volume operators considered on appropriately chosen domains. On the other hand, the operators in finite volume have a natural choice of monotone eigenvalue branches that retain the same monotonicity and ordering properties as the corresponding diagonal entries (see \cite[Remark 4.5]{Ilya}). Due to smallness of the off-diagonal entries $H(x,x_0;\bze,\bn)$, every diagonal entry of the partial diagonalization procedure is close to some eigenvalue of some finite volume operator. If one establishes that the said eigenvalue is from the ``correct'' branch, that would provide the missing  intrinsic monotonicity property. This follows from the fact that (ind4) on the previous scale implies a ``coarse'' ordering of the remaining eigenvalues, sufficient to verify that the eigenvalue under consideration is located at the correct position.

The following finite volume eigenvalue parametrization is, essentially, established in \cite{Ilya}, see also \cite[Proposition 3.1]{JK2}.
\begin{prop}
\label{prop_box_eigenvalues}
Let $\Lambda\subset \Z^d$ be a finite subset, and suppose that $f\colon \R\to [-\infty,+\infty)$ is $1$-periodic  and non-decreasing on $[0,1)$. For every $x\in \R$, there exists a unique numeration of the eigenvalues of $H_{\Lambda}(x)$ by the lattice points $\bn\in \Lambda$ such that 
\bee
\label{eq_eigenvalue_ordering}
E_{\bm}(x)\le E_{\bn}(x)\quad\text{if and only if}\quad\{x+\bm\cdot\omega\}\le \{x+\bn\cdot\omega\}.
\ene
Each eigenvalue $E_{\bn}(x)$, considered as a function of $x$, is non-decreasing on $[-\bn\cdot\omega,-\bn\cdot\omega+1)$. If $f$ is strictly increasing on $[0,1)$ or Lipschitz monotone, then the same holds for $E_{\bn}$.
\end{prop}
\begin{rem}
For $\ep=0$, these eigenvalue branches become the diagonal entries: $E_{\bn}(x)=f(x+\bn\cdot\omega)$. It is important for the proof that the off-diagonal entries of $H_{\Lambda}(x)$ are constant in $x$.
\end{rem}
\begin{rem}
\label{rem_locally_monotone}
If $f$ is $\alpha$-H\"older monotone, then the functions $E_{\bn}(x)$ are locally $\alpha$-H\"older monotone (that is, between the discontinuity points $\{-\bn\cdot\omega\}$, $\bn\in \Lambda$). Unlike Lipschitz monotonicity, $\alpha$-H\"older monotonicity with $\alpha>1$ on a large interval is not necessarily inherited from the same property on smaller intervals. However, using the inequality
$$
x^{\alpha}+y^{\alpha}\ge 2^{1-\alpha}(x+y)^{\alpha},\quad x,y\ge 0,\quad \alpha\ge 1,
$$
one easily show the following: suppose that $f$ is $\alpha$-H\"older monotone and $(a,b)$ is an interval containing at most one discontinuity point $\{-\bn\cdot\omega\}$, $\bn\in \Lambda$. Suppose that $E_{\bn}$ is an eigenvalue branch, considered in Proposition \ref{prop_box_eigenvalues}, that is non-decreasing on $(a,b)$. Then
$$
E_\bn(y)-E_\bn(x)\ge 2^{1-\alpha}(y-x)^{\alpha},\quad \text{for}\quad a<x<y<b.
$$
\end{rem}
Recall that (ind4) at the stage $s$ was stated as
\bee
\label{eq_ind4_s}
f^{(s)}(y)-f^{(s)}(x)\ge 2^{1-\alpha}(y-x)^{\alpha}-4^d\l(1-\frac{1}{s-k+2}\r)\ep^{k}\magn(k),
\ene
assuming (without loss of generality)
\bee
\label{eq_same_assumptions}
\quad x\in I_s,\,y\in I_k,\quad 0\le k\le  s, \quad 0\le x<y<1.
\ene
We would like to establish the same estimate at the stage $s+1$, assuming that (ind1) -- (ind3) are already established at that stage. By combining \eqref{eq_ind4_s} with (ind3) at the stage $s+1$, it is easy to see that
\bee
\label{eq_ind4_s2}
f^{(s+1)}(y)-f^{(s+1)}(x)\ge  2^{1-\alpha}(y-x)^{\alpha}-4^d\l(1-\frac{1}{(s+1)-k+2}\r)\ep^{k}\magn(k)
\ene
under the same assumptions \eqref{eq_same_assumptions} since, by (ind3) that is already established at the stage $s+1$, we have
$$
|f^{(s+1)}(x,x_0)-f^{(s)}(x,x_0)|\le \ep^{2s+1}\magn(2s+1)\ll \l(\frac{1}{s-k+2}-\frac{1}{(s+1)-k+2}\r)\ep^{k}\magn(k).
$$
Inequality \eqref{eq_ind4_s2} is ``almost'' (ind4) at the stage $s+1$, except the assumptions \eqref{eq_same_assumptions} do not include the case $x,y\in I_{s+1}$, where a (crucial) improvement in the power of $\ep$ in the right hand side is needed. In other words, it remains to establish
\bee
\label{eq_ind4_to_establish}
f^{(s+1)}(y,x_0)-f^{(s+1)}(x,x_0)\ge 2^{1-\alpha}(y-x)^{\alpha}-\frac{4^d}{2} \ep^{s+1}\magn(s+1);\quad x,y\in I_{s+1};\,\,0\le x<y<1.
\ene
Let $x\in I_{s+1}$. From Remark \ref{rem_lipschitz}, we have
\bee
\label{eq_no_smallden2}
\frac{f^{(s+1)}(x+\bn\cdot\omega,x_0)-f^{(s+1)}(x,x_0)}{{\mathrm{sgn}(\{x\}-\{x+\bn\cdot\omega\})\l|\{x\}-\{x+\bn\cdot\omega\}\r|^{\alpha}}}\ge \frac{1}{2^{\alpha}}, \quad \text{for}\quad |\bn|\le 3s,\,x\in I_s.
\ene
Following Remark \ref{rem_extended_boxes}, denote
$$
\Lambda_s=\Lambda_s(x_0):=\bigcup_{x\in I_s}\Lambda^{(s)}_{\bze}(x,x_0).
$$
The following is the main result of this section.
\begin{thm}
\label{th_ind_4}
Assume that {\rm (ind1) -- (ind3)} are established at the stages up to $s+1$, and {\rm (ind4)} up to stage $s$. Define $\Lambda_{s+1}$ as above and $E_{\bze}(x)=E_{\bze}(x,x_0)$ associated to $H_{\Lambda_{s+1}}(x)$ as in Proposition $\ref{prop_box_eigenvalues}$. Then
\bee
\label{eq_distance_eigenvalue}
|f^{(s+1)}(x,x_0)-E_{\bze}(x)|\le 3^d \ep^{s+2}\magn(s+2),\quad\text{for}\quad x\in I_{s+1}. 
\ene
\end{thm}
In view of Proposition \ref{prop_box_eigenvalues}, the conclusion of Theorem \ref{th_ind_4} completes the proof of \eqref{eq_ind4_to_establish}, since the interval $I_s$ is small enough to contain at most one discontinuity point (see Remark \ref{rem_locally_monotone}). As a consequence, it also establishes (ind4) at stage $s+1$.
\begin{proof}
We will first establish
\bee
\label{eq_distance_spectrum}
\dist\l((f^{(s+1)}(x,x_0),\sigma(H_{\Lambda_{s+1}}(x))\r)\le 3^d \ep^{s+2}\magn(s+2).
\ene
Indeed, $f^{(s+1)}(x,x_0)$ becomes an exact eigenvalue of $H_{\Lambda_{s+1}}^{(s+1)}(x,x_0)$ if one removes all off-diagonal entries of the form $H^{(s+1)}(x,x_0;\bze,\bn)$, $\bn\neq\bze$. Since $x\in I_{s+1}$ and due to (ind1) -- (ind3), these entries are of the order (at least) $s+2$, which implies the estimate in the right hand side of \eqref{eq_distance_spectrum} by the norm of the corresponding perturbation. Note that $H_{\Lambda_{s+1}}^{(s+1)}(x,x_0)$ is unitarily equivalent to $H_{\Lambda_{s+1}}(x)=H_{\Lambda_{s+1}}^{(0)}(x)$, due to the structure of the support of $W^{(s+1)}(x,x_0)$ described in Lemma \ref{lemma_W_support}.

In other words, $f^{(s+1)}(x,x_0)$ is very close to {\it some} eigenvalue of $H_{\Lambda_{s+1}}(x)$. Let 
\begin{multline*}
N_+(x)=N_+(x,x_0):=\#\l\{\bm\in\Lambda_{s+1}\colon \{x+\bm\cdot\omega\}>\{x\}\r\}=\\=\#\l\{\bm\in\Lambda\colon f^{(s+1)}(x+\bm\cdot\omega,x_0)>f^{(s+1)}(x,x_0)\r\},	
\end{multline*}
\begin{multline*}
N_-(x)=N_-(x,x_0):=\#\l\{\bm\in\Lambda_{s+1}\colon \{x+\bm\cdot\omega\}<\{x\}\r\}=\\=\#\l\{\bm\in\Lambda\colon f^{(s+1)}(x+\bm\cdot\omega,x_0)<f^{(s+1)}(x,x_0)\r\}.	
\end{multline*}
The equalities in both definitions follow from \eqref{eq_no_smallden2}. In order to establish \eqref{eq_distance_eigenvalue}, it is sufficient to obtain the following two estimates:
\bee
\label{eq_counting_1}
\#\l(\sigma(H_{\Lambda_{s+1}}(x))\cap (E_{\bze}(x)+3^d \ep^{s+2}\magn(s+2),+\infty)\r)\ge N_+(x),\quad x\in I_{s+1};
\ene
\bee
\label{eq_counting_2}
\#\l(\sigma(H_{\Lambda_{s+1}}(x))\cap [-\infty,E_{\bze}(x)-3^d \ep^{s+2}\magn(s+2)\r)\ge N_-(x),\quad x\in I_{s+1}.
\ene
In other words, in order to show that an eigenvalue of $H_{\Lambda_{s+1}}(x)$ is equal to $E_{\bze}(x)$, it is sufficient to show that there are at least $N_+(x)$ eigenvalues above it and at least $N_-(x)$ eigenvalues below it (see \eqref{eq_eigenvalue_ordering}). Note that both inequalities in \eqref{eq_counting_1} and \eqref{eq_counting_2} would actually be equalities, since the total number of eigenvalues is $N_+(x)+N_-(x)+1$.

Due to symmetry, it is sufficient to establish \eqref{eq_counting_1}. Let
\begin{multline*}
L_k(x)=L_k(x,x_0):=\{\bn\in \Lambda_{s+1}\colon \{x+\bn\cdot\omega\}\ge \{x\}; \{x+\bn\cdot\omega\}\in I_k\setminus I_{k+1}\},\\ k=0,1,\ldots,s,s+1,\quad x\in I_{s+1}.	
\end{multline*}
For the convenience of notation in the above definition for $k=s+1$, we will assume $I_{s+2}=\varnothing$. The choice of small $\beta$ will also guarantee that $L_{s+1}=\{\bze\}$. Assuming that $s\ge 2$, we can impose an additional requirement on the smallness of $\beta$ in order to guarantee $L_{s}=\varnothing$. In the sequel, we will assume $s\ge 2$. Note that $L_k(x)$ are disjoint, and
\bee
\label{eq_dimensions}
\#(L_0(x)\cup\ldots \cup L_{s-1}(x))=N_+(x).
\ene
Let
$$
m_k(x)=m_k(x,x_0):=\min\{f^{(s+1)}(x+\bn\cdot\omega,x_0)\colon \bn\in L_0(x)\cup\ldots\cup L_k(x)\}.
$$
Note that, from \eqref{eq_no_smallden2}, we have a lower bound
\bee
\label{eq_mk_lower}
m_k(x)\ge f^{(s+1)}(x,x_0)+\l\{\frac12 \|\bn\cdot\omega\|\colon \bn\in L_0(x)\cup\ldots\cup L_s(x)\r\}.
\ene
We claim the following:
\bee
\label{eq_claim_counting}
\inf\sigma(H_{L_0\cup\ldots \cup L_k(x)}^{(s+1)}(x,x_0))\ge \max_{r\colon 0\le r\le k}\{m_r(x)-3^d \ep^{r+1}\magn(r+1)\}.
\ene
Let
$$
\sigma_k(x):=\sigma(H_{L_0(x)\cup\ldots \cup L_k(x)}^{(s+1)}(x,x_0)).
$$
It will be convenient to consider $\sigma_k(x)$ as a multi-set (so that the eigenvalues are counted with multiplicities). For two finite multi-sets of the same cardinality
$$
A=\{a_1,\ldots,a_N\}, \quad B=\{b_1,\ldots,b_N\},\quad a_j,b_j\in \R,\quad j=1,\ldots,N,
$$
assuming that they are ordered:
$$
a_1\le a_2,\ldots\le a_N,\quad b_1\le b_2,\ldots\le b_N,
$$
define
$$
\dist(A,B):=\max\limits_{1\le j\le N}|a_j-b_j|.
$$
In the above notation, we have
$$
\dist\l(\sigma_0(x),\{f^{(s+1)}(x+\bn\cdot\omega,x_0)\colon \bn\in L_{0}(x)\}\r)\le 2\cdot 2^d\ep\,\magn(1),
$$
since if one removes the off-diagonal part of $H^{(s+1)}_{L_0(x)}(x,x_0)$, the operator will become diagonal with entries $\{f^{(s+1)}(x+\bn\cdot\omega,x_0)$, and the right hand side is an estimate of the norm of the off-diagonal part. As a consequence,
$$
\inf \sigma_0(x)\ge m_0(x)-2\cdot 2^d\ep\,\magn(1).
$$
In general, $\sigma_{k+1}(x)$ is obtained from the union
\bee
\label{eq_union}
\sigma_k(x)\cup  \{f^{(s+1)}(x+\bn\cdot\omega)\colon \bn\in L_{k+1}(x)\},
\ene
by applying a perturbation of size at most $2\cdot 2^d\ep^{k+2}\,\magn(k+2)$ to each element, so that
$$
\dist\l(\sigma_{k+1}(x),\sigma_k(x)\cup  \{f^{(s+1)}(x+\bn\cdot\omega)\colon \bn\in L_{k+1}(x)\}\r).
$$
 In this case, the perturbation accounts for off-diagonal matrix elements between $L_{k+1}$ and $L_0\cup\ldots\cup L_k$, as well as those between points of $L_k$. In both cases, the order of these off-diagonal entries is at most $k+2$ (one can also incorporate the case $k=0$ into this scheme by assuming $\sigma_{-1}=\varnothing$).

As a consequence, in order to arrive at $\sigma_s(x)$, each element $f^{(s+1)}(x+\bn\cdot\omega)$ with $\bn\in L_j(x)$ that appears at some stage in \eqref{eq_union}, will undergo a perturbation of size at most
$$
\sum_{r\ge j}2\cdot 2^d \ep^{r+1}\magn(r+1)\le 3^d \ep^{j+1}\magn(j+1).
$$
Combining it with \eqref{eq_no_smallden2} and using the lower bound \eqref{eq_mk_lower}, we obtain the following:
\begin{multline*}
\inf \sigma(H_{L_0\cup\ldots \cup L_s}^{(s+1)}(x))\ge f^{(s+1)}(x)+\min_{0\le j\le s}\min_{\bn\in L_j}\l\{\frac12 \|\bn\cdot\omega\|-3^d \ep^{j+1}\magn(j+1)\r\}\ge \\
\ge f^{(s+1)}(x)+\min_{0\le j\le s-1}\l\{\frac12 \beta^{\frac{j+1}{\log (j+2)}}-3^d \ep^{j+1}\magn(j+1)\r\}.
\end{multline*}
Clearly, one can guarantee 
$$
\beta^{\frac{j+1}{\log (j+2)}}\gg 3^d \ep^{j+1}\magn(j+1),\quad\text{for}\quad  j=0,1,2,\ldots
$$
say, by taking $\ep\ll M^{-1}$. Under this assumption, the minimal value will be attained at $j=s-1$. 
In view of \eqref{eq_dimensions}, we have produced a subspace $\ell^2(L_0\cup\ldots\cup L_k)$ such that the restriction of $H^{(s+1)}_{\Lambda_{s+1}}(x)$ into this subspace has $N_+(x)$ eigenvalues greater than $f^{(s+1)}(x)+3^d \ep^{j+1}\magn(j+1)$. This completes the proof.
\end{proof}
In order to establish (conv6) later in the next section, we will state a more general (and more cumbersome) version of Theorem \ref{th_ind_4} in the form of a corollary. The reader can check that the proof is the same. Suppose that $I\subset I_{s+1}$ is an interval. As discussed in Remark \ref{rem_extended_boxes}, denote extended regions associated to $I$:
$$
\Lambda_{\bn}^{(s)}=\Lambda_{\bn}^{(s)}(x_0):=\bigcup_{x\in I}\Lambda_{\bn}^{(s)}(x,x_0),
$$
where $\Lambda_{\bn}^{(s)}(x,x_0)$ are maximal extended regions \eqref{eq_W_support}. Since $I\subset I_s$, the regions $\Lambda_{\bn}^{(s)}$ will not overlap (see Remark \ref{rem_extended_boxes}).
\begin{cor}
\label{cor_extended_ind4}
Let $\Lambda\subset \Z^d$ be a finite subset satisfying the following properties:
\bee
\label{eq_condition_domain}
\Lambda_{\bze}^{(s)}\subset \Lambda\subset \Z^d\setminus\bigcup_{\bm\in\Z^d\setminus\{\bze\}}\Lambda_{\bm}^{(s)},
\ene
and
\bee
\label{eq_condition_interval}
\|\bn\cdot\omega\|\ge 10|I_s|,\quad \forall \bn\in \Lambda.
\ene
Assume that {\rm (ind1) -- (ind3)} are established at the stages up to $s+1$, and {\rm (ind4)} up to stage $s$. Then
$$
|f^{(s+1)}(x,x_0)-E_{\bze}(x)|\le 3^d \ep^{s+2}\magn(s+2),\quad\text{for}\quad x\in I. 
$$
\end{cor}
\begin{proof}
As mentioned above, the proof repeats the proof of Theorem \ref{th_ind_4}, up to notation. We will note that the assumption \eqref{eq_condition_interval} is used in order to state that $L_s=\varnothing$ and $L_{s+1}=\{\bze\}$, and \eqref{eq_condition_domain} required to state that $H_\Lambda^{(s)}(x,x_0)$ is unitarily equivalent to $H_\Lambda(x,x_0)=H_\Lambda^{(0)}(x,x_0)$.
\end{proof}
\section{Proof of Theorem \ref{th_convergence}}\label{section_convergence}
\subsection{Proof of (conv3) -- (conv5)}
Once the proof of Theorem \ref{th_inductive_estimates} is completed, it is easier to restate the estimates (ind1) -- (ind3) as follows.
\begin{cor}
\label{cor_inductive_simple}Assume that $\ep$ is small in the sense of Remark $\ref{rem_small_parameters}$. Then
$$
\begin{cases}
|U^{(s)}(x,x_0;\bn,\bn)-1|\le \ep^{(1-\delta/9)s};\\
|U^{(s)}(x,x_0;\bm,\bn)|\le \ep^{(1-\delta/9)s},&0<|\bm-\bn|\le s;\\
|U^{(s)}(x,x_0;\bm,\bn)|\le \ep^{(1-\delta/9)|\bm-\bn|},&s<|\bm-\bn|\le s(1+\delta);\\
U^{(s)}(x,x_0;\bm,\bn)=0,&|\bm-\bn|>s(1+\delta/9).
\end{cases}
$$
As a consequence,
\bee
\label{eq_W_estimate}
|W^{(s)}(x,x_0;\bn,\bn)-1|\le \ep^{(1-\delta/8)},\quad |W^{(s)}(x,x_0;\bm,\bn)|\le \ep^{(1-\delta/8)|\bm-\bn|},
\ene
\end{cor}
\begin{proof}
Follows directly from the inductive estimates on $H^{(s)}(x,x_0)$ and the definition of $U^{(s)}(x,x_0)$, since one can absorb all additional factors into $\ep^{-\delta s/100}$ or $\ep^{-\delta|\bm-\bn|/100}$, respectively. The estimate \eqref{eq_W_estimate} follows the definition of $W^{(s)}(x,x_0)$ by estimating the matrix elements of the product.
\end{proof}
\begin{cor}
Under the assumptions of Theorem $\ref{th_convergence}$, the limit
$$
W^{(\infty)}(x,x_0)=\lim\limits_{s\to +\infty}W^{(s)}(x,x_0)
$$
exists in the operator norm topology, and satisfies
$$
\|W^{(\infty)}(x,x_0)-\one\|\le \ep^{1-\delta/7}.
$$
Moreover,
$$
(W^{(\infty)}(x_0,x_0)^{-1}H(x)W^{(\infty)}(x_0,x_0))(\bze,\bn)=0,\quad \bn\neq \bze.
$$
As a consequence,
$$
H(x)W^{(\infty)}(x_0,x_0)e_{\bze}=E(x_0)W^{(\infty)}(x_0,x_0)e_{\bze},
$$
where $E(x_0)\in \R$.
\end{cor}
\begin{proof}
The first two claims follow from the fact that $\|U^{(s)}(x,x_0)-\one\|\le \ep^{(1-\delta/9)s}$. The last two claims follow from passing to the limit in (ind1) and (ind2).
\end{proof}
This completes the proof of (conv3) -- (conv5). It is easy to see that Corollary \ref{cor_localization} follows from covariance and an observation that if one chooses a normalization
$$
\|\psi(x_0)\|_{\ell^2(\Z)}=1,\quad \psi(x_0;\bze)>0,
$$
then the set
\bee
\label{eq_orthogonal_basis}
\{\psi_{\bn}(x_0):=T^{\bn}\psi(x_0-\bn\cdot\omega)\colon \bn\in \Z^d\}
\ene
is a small perturbation of the standard basis. More precisely, if one considers a matrix
$$
A(x_0;\bm,\bn):=\psi_{\bn}(x_0;\bm),
$$
then, by Schur test, it will satisfy
$$
\|A(x_0)-\one\|\le \ep^{1-\delta/6},
$$
which, together with the fact that \eqref{eq_orthogonal_basis} must be an orthogonal collection of eigenvectors of $H(x_0)$, establishes completeness of that set. 
\subsection{Monotonicity of the eigenvalue function and the proof of (conv6)} We will start from proving a weaker version of (conv6): that $E(x_0)$ is non-decreasing on $[0,1)$ and is Lipschitz monotone provided $f$ was Lipschitz monotone. We will use the following elementary result.
\begin{lem}
\label{lemma_gamma_monotone}
Let $\{\ep_j\}_{j=0}^{+\infty}$, $\{\delta_j\}_{j=0}^{+\infty}$ be two sequences of positive real numbers such that
$$
\ep_j\to 0,\quad \delta_j\to 0,\quad \frac{\ep_j}{\delta_j}\to 0\quad \textrm{as}\quad j\to +\infty,
$$
and let $f\colon \R\to [-\infty,+\infty)$ be a $1$-periodic function such that, for every $x_0\in \R$ and $j\in \N$ there exists a $1$-periodic function $g_{j,x_0}\colon \R\to [-\infty,+\infty)$ that is strictly increasing on $[0,1)$ satisfying
\bee
\label{eq_partial_gamma_monotonicity}
|f(x)-g_{j,x_0}(x)|<\ep_j,\quad x\in (x_0-\delta_j,x_0+\delta_j).
\ene
Then $f$ is non-decreasing on $[0,1)$. If all functions $g_{j,x_0}$ are Lipschitz monotone on $[0,1)$, then $f$ is also Lipschitz monotone on $[0,1)$.
\end{lem}
\begin{rem}
If $g_{j,x_0}(x)=-\infty$, then \eqref{eq_partial_gamma_monotonicity} is understood as $f(x)=-\infty$. Due to the assumptions on $g_{j,x_0}$, this can only happen at $x\in \Z$. This is the only place where strict monotonicity of $g_{j,x_0}$ is used, and it can be replaced by $g_{j,x_0}$ being finite on $(0,1)$.
\end{rem}
\begin{proof}
Note that \eqref{eq_partial_gamma_monotonicity} implies that $f(x)$ must be finite except maybe for integer points. Let $0<x<y<1$, and let $n_j$ be the largest integer number such that $n_j\delta_j\le y-x$. By repeatedly applying \eqref{eq_partial_gamma_monotonicity} and the triangle inequality, we arrive to
$$
f(y)-f(x)\ge -(n_j+1)\ep_j\ge -(y-x)\frac{\ep_j}{\delta_j}-\ep_j\to 0,\quad j\to +\infty.
$$
Similarly, if $g_{j,x_0}$ are Lipschitz monotone, we have
$$
f(y)-f(x)\ge y-x-(n_j+1)\ep_j\ge y-x-(y-x)\frac{\ep_j}{\delta_j}-\ep_j=(y-x)\l(1-\frac{\ep_j}{\delta_j}\r)-\ep_j,
$$
which implies Lipschitz monotonicity of $f$ after taking $j\to +\infty$.
\end{proof}

{\noindent\bf The weaker version of (conv6).} Suppose that $\beta_0$ is chosen to satisfy the assumptions of (conv1) -- (conv5) of Theorem \ref{th_convergence} and $0<\beta<\beta_0$, $\tilde\beta:=\beta/3$, and $\ep_0$ is chosen such that the conclusion of (conv1) -- (conv5) is satisfied both with $\beta$ and $\tilde\beta$. As a consequence, both of the inductive schemes will converge to the same eigenvalue and the same eigenfunction (up to normalization) of the operator $H(x_0)$ (since each eigenfunction will be a small perturbation of $e_{\bze}$). Denote the corresponding finite volume operators by $H^{(s)}$, $\tilde H^{(s)}$, with the diagonal entries $f^{(s)}(x,x_0)$ and $\tilde f^{(s)}(x,x_0)$, respectively, and the corresponding values of $\beta_j$ by $\beta_j$, $\tilde\beta_j$. Fix some $x_0\in [0,1)$ and suppose that $|x_1-x_0|<\frac13 \beta^{(s)}$. Then we have the following inclusion of the intervals:
\bee
\label{eq_inclusion_intervals}
\l(x_1-\tilde\beta_j,x_1+\tilde\beta_j\r)\subset (x_0-\beta_j,x_0+\beta_j),\quad\text{for}\quad  0\le j\le s.
\ene
From Claim (5) in Lemma \eqref{eq_W_support}, and Remark \ref{rem_extended_boxes}, we have that the extended regions $\Lambda_{\bze}^{(s)}(x_0)$, constructed for the operators $H^{(s)}(x,x_0)$, would contain the corresponding extended regions constructed for $\tilde H^{(s)}(x,x_1)$, assuming 
$$
|x_1-x_0|<\frac13 \beta^{(s)},\quad |x-x_1|<\tilde\beta^{(s)}.
$$
For $\ep$ small enough, the region $\Lambda_{\bze}^{(s)}(x_0)$ will satisfy the assumptions of Corollary \ref{cor_extended_ind4}. Let $E_{\bze}(x,x_0)$ be the $1$-monotone eigenvalue branch obtained from applying Proposition \ref{prop_box_eigenvalues} to $\Lambda=\Lambda_{\bze}^{(s)}(x_0)$. We arrive at
\bee
\label{eq_approximate_x1}
|\tilde f^{(s)}(x_1,x_1)-E_{\bze}(x_1,x_0)|\le 3^d \ep^{s+2}\magn(s+2),\quad \text{for}\quad x_1\in I= \l(x_0-\beta^{(s)}/3,x_0+\beta^{(s)}/3\r).
\ene
From (ind3), we also have
$$
|\tilde f^{(s)}(x_1,x_1)-E(x_1)|\le 2\ep^{2s}\magn(2s),\quad \text{for} \quad x_1\in \R,
$$
where
$$
E(x_1)=\lim\limits_{s\to +\infty}f^{(s)}(x_1,x_1)=\lim\limits_{s\to +\infty}\tilde f^{(s)}(x_1,x_1)
$$
is the limiting eigenvalue function, whose existence has been established in (conv1) -- (conv5). In the above estimates, $\magn$ is chosen with respect to $\tilde\beta^{(s)}$. By combining the above, we arrive to
$$
|E(x_1)-E_{\bze}(x_1,x_0)|\le 3^{d+1} \ep^{s+2}\magn(s+2),\quad \text{for}\quad x_1\in I= \l(x_0-\beta^{(s)}/3,x_0+\beta^{(s)}/3\r).
$$
Since $E_{\bze}(x_1,x_0)$ associated to $\Lambda_{\bze}^{(s)}(x_0)$ is monotone (or Lipschitz monotone) as a function of $x_1\in \R$ and
$$
\frac{3^d \ep^{s+2}\magn(s+2)}{\beta^{(s)}}\to 0,\quad \text{as}\quad s\to +\infty
$$
assuming $\ep$ is small in the sense of Remark \ref{rem_small_parameters}, we have that $E(x_1)$ is also either monotone or Lipschitz monotone as a function of $x_1$ due to Lemma \ref{lemma_gamma_monotone}.\,\,\qed

Finally, we can establish the remaining part of (conv6) as a consequence of the general result \cite{CS_ids}.
\begin{prop}
\label{prop_ids}
Under the assumptions of Theorem $\ref{th_main}$, the inverse function of $E(\cdot)|_{[0,1)}$ is equal to the integrated density of states. As a consequence, $E(\cdot)$ cannot vanish on an interval.
\end{prop}
\begin{proof}
The relation between $E(\cdot)$ and the IDS is the following standard calculation. Denote by $dN(E)$ the density of states measure of the ergodic operator family $\{H(x_0)\}_{x_0\in \T}$, and suppose that $g\in C_0^{\infty}(\R)$. Let
$$
E_{\bn}(x):=E(x-\bn\cdot\omega),\quad \psi_{\bn}(x)=T^{\bn}\psi(x-\bn\cdot\omega)
$$
be the eigenvalues and eigenvectors of $H(x)$. Then, by the definition of the density of states measure, change of variable, and normalization of eigenvectors, we have
\begin{multline*}
\int g(E)\,dN(E)=\int_0^1 \<g(H(x)e_{\bze},e_{\bze}\>\,dx=\int_{[0,1)}\sum_{\bn\in \Z^d}g(E_{\bn}(x))|\psi_{\bn}(x;\bze)|^2\,dx\\
=\int_{[0,1)}\sum_{\bn\in \Z^d}g(E(x-\bn\cdot\omega))|\psi(x-\bn\cdot\omega;\bn)|^2\,dx\\
=\int_{[0,1)}\sum_{\bn\in \Z^d}g(E(x))|\psi(x;\bn)|^2\,dx=\int_{[0,1)}g(E(x))\,dx,
\end{multline*}
which implies the first claim. We have already shown that $E$ is non-decreasing on $[0,1)$. The only possibility for it to be not strictly increasing would be to vanish on an interval. However, the latter would lead to an atom in the measure $dN(E)$, which is impossible due to the general results \cite{CS_ids} (see also \cite{BK}). In fact, the cited results show that $dN(E)$ will always be $\log$-H\"older continuous, from which one can extract some quantitative monotonicity bounds on $E(\cdot)$ which are weaker than the $\alpha$-H\"older property.
\end{proof}
\section{Existence of spectral gaps}\label{section_gaps}
In this section, we prove Lemma \ref{lemma_gaps}, which will immediately imply Theorem \ref{th_gaps}. Note that some parts of Lemma \ref{lemma_gaps} are merely references to known facts. We will often refer to some constructions from \cite{Cantor}, or some well known constructions from other sources that are used in \cite{Cantor}.

Suppose that $f$ is sawtooth-type. As discussed in the introduction, let
$$
H_{\bze,t}:=H_{\bze}+t\<e_{\bze},\cdot\>e_{\bze},\quad H_t(0):=H(0)+t\<e_{\bze},\cdot\>e_{\bze}.
$$
In the setting of gap filling, it is natural to include $t=\infty$ into consideration (see \cite{Cantor} and \cite{Simon}) by considering infinite coupling. We will denote $\rbar:=\R\cup\{\infty\}$, with the topology in which it is homeomorphic to a circle. Let also
$$
H_{\bze}^{\perp}:=(H(0)-H_{\bze})|_{\ell^2_{\bze}(\Z^d)^{\perp}}
$$
be the remaining part of the operator $H(0)$. Note that
$$
H_t(0)=H_{\bze,t}\oplus H_{\bze}^{\perp}
$$
for all $t\in \R$ (in other words, $H_{\bze}^{\perp}$ naturally does not depend on $t$). The operator $H_t(0)$ can be obtained from the original operator $H(0)$ if one replaces $f$ by
$$
f_t(x)=\begin{cases}
f(x),&x\in (0,1)\\
t,&x=0,
\end{cases}
$$
extended periodically into $\R$. For $t\in\mathcal A:=\rbar\setminus(f(0),f(1-0))$, the function $f_t$ is $1$-monotone either on $[0,1)$ or on $(0,1]$ (see Remark \ref{rem_infinite}). As a consequence, Theorem \ref{th_main} applied to $H_t(0)$ implies that the spectrum of $H_{\bze}(t)$ is purely point. Let 
$$
\Sigma_{\bze}:=\sigma(H_{\bze}),\quad \Sigma^{\perp}_{\bze}:=\sigma(H^{\perp}_{\bze}).
$$
The proof of Theorem \ref{th_gaps} is contained in the following lemma. Note that a significant part of it (for example, (4)) is a restatement of known results.
\begin{lem}
\label{lemma_gaps}
Suppose that $f$ is $1$-periodic and $\alpha$-H\"older monotone on $[0,1)$. Define $f_t$, $H_{\bze,t}$, $H_t(0)$ as above. Suppose that $\omega$ is Diophantine and $\ep$ is small enough so that the conclusion of Theorem \ref{th_main} holds. Then
\begin{enumerate}
	\item We have $\sigma_{\mathrm{ess}}(H(t))=\Sigma_{\bze}\cup \Sigma^{\perp}_{\bze}$ for all $t\in \R$ and 
	$\sigma(H(t))=\Sigma_{\bze}\cup \Sigma^{\perp}_{\bze}$ for $t\in [f(0),f(1-0)]$.
	\item Every isolated eigenvalue of $H_{\bze,t}$ with $t\in \mathcal A$ is also an isolated eigenvalue of $H_t(0)$ with $t\in \mathcal A$, and vice versa. In both cases, this eigenvalue is a non-constant real analytic function in some neighborhood of $t$. As a consequence, $H_{\bze,t}$ also has no isolated eigenvalues for $t\in [f(0),f(1-0)]$, and $\sigma_{\mathrm{ess}}(H_{\bze,t})=\Sigma_{\bze}$.
	\item For every gap $G$ of $\Sigma_{\bze}$ (that is, a connected component of $\rbar\setminus \Sigma_{\bze}$), we have 
\bee
\label{eq_gap_kinds}
\text{either}\quad G\cap\bigcup_{t\in \rbar}\sigma(H_{\bze,t})=\varnothing,\quad\text{or}\quad G\cap\bigcup_{t\in \mathcal A}(\sigma(H_{\bze,t}))=G.
\ene
\item For a dense $G_{\delta}$-subset of values of $t\in \mathcal A$, all eigenvalues of $H_{\bze,t}$ are isolated. As a consequence, there are infinitely many gaps of second kind in \eqref{eq_gap_kinds}. Each of these gaps is also a gap in $\Sigma_{\bze}\cup\Sigma^{\perp}_{\bze}$.
\end{enumerate}
\end{lem}
\begin{proof}
For sawtooth-type potentials, the first claim in (1) is, essentially, well-known, (see \cite[proof of Lemma 2.8]{Cantor}; it is easy to see that it works for all monotone potentials, not just Lipschitz monotone). The second claim in (1) follows from gap filling, see \cite[Theorem 1.2]{Cantor}, see also \cite[Section 4.1]{Cantor} and Chapter 17 of \cite{Simon}. Note that gap filling implies that eigenvalues in gaps of $\sigma_{\mathrm{ess}}(H(t))$ can only appear for $t\in \mathcal A$.

In order to prove (2), note that, for $t\in \mathcal A$, the spectra of $H_{\bze,t_0}$ and $H_t(0)$ are simple. 
If $H_{\bze,t_0}$ has an isolated eigenvalue $\lambda$, it extends to a real analytic function $\lambda(t)$ defined in a small neighborhood of $t_0$, and there also exists an analytic branch $\psi(t)$ of the corresponding normalized eigenvector \cite{Kato}. The fact that $\lambda(t)$ is not constant, follows from the well-known Feynman--Hellman formula:
$$
\frac{d}{dt}\lambda(t)=|\psi(t;\bze)|^2
$$
and the fact that the eigenvectors of $H_{\bze,t}$ cannot vanish at $\bze$. The range of $\lambda(t)$ cannot intersect $\sigma(H_{\bze}^{\perp})$, since it would contradict simplicity of spectra. As a consequence, $\lambda(t)$ will be an isolated eigenvalue of $H_t(0)$. Conversely, if $\lambda$ is an isolated eigenvalue of $H_t(0)$, one can also find $\lambda(t)$ and $\psi(t)$ as above. If $\psi(t;\bze)=0$, then $\lambda(t)$ is an isolated eigenvalue of $H_t(0)$ for all values of $t$, which contradicts (1). If $\psi(t;\bze)\neq 0$, then $\psi(t)$ is also an eigenvector of $H_{\bze,t}$, and the claim follows.

Part (3) follows from the general gap filling theory for rank one perturbations, see (as mentioned before) \cite[Section 4.1]{Cantor} and Chapter 17 of \cite{Simon}.

Part (4) is, essentially, the del Rio--Gordon--Makarov--Simon theorem, see \cite{drms,G1,G2}. From Theorem \ref{th_main}, we have that $\sigma(H_{\bze,t})$ is pure point for $t\in \mathcal A$. Therefore, for a dense $G_{\delta}$-subset of these $t$, $H_{\bze,t}$ has infinitely many isolated eigenvalues, each of which must be contained in a distinct gap.
\end{proof}
\begin{rem}
\label{rem_continuity}
Suppose that $f$ satisfies the assumptions of Theorem \ref{th_main}. Following \cite{Cantor}, we will say that $f$ is {\it Maryland-type} if $f|_{[0,1)}$ extends to a continuous map between $[0,1)$ and $[-\infty,+\infty)$, where both sets are identified with circles in the natural topology. It was shown in \cite{Ilya,Cantor} that operators \eqref{eq_h_def} with Maryland-type potentials do not have gaps in their spectra. 

In Theorem \ref{th_gaps}, we assume for simplicity that $f$ is sawtooth-type (that is, has no discontinuities on $(0,1)$), since we are using the gap filling result \cite[Theorem 1.2]{Cantor} which is only stated for potentials of this kind. By considering an appropriate modification of \cite[Theorem 1.2]{Cantor}, one can extend the conclusion of Theorem \ref{th_gaps} to all potentials within the class considered in Theorem \ref{th_main} that are not Maryland-type.
\end{rem}
\end{document}